\documentclass[aos]{imsart}

\RequirePackage{amsthm,amsmath,amsfonts,amssymb}
\RequirePackage[numbers]{natbib}
\RequirePackage[colorlinks,citecolor=blue,urlcolor=blue]{hyperref}
\RequirePackage{graphicx}% uncomment this for including figures
\usepackage{stmaryrd}
\usepackage{dsfont}
\usepackage{eucal}
\usepackage{bm}
\usepackage{varioref}

\startlocaldefs
\newtheorem{thm}{Theorem}
\newtheorem{prop}[thm]{Proposition}

\newtheorem{lemma}[thm]{Lemma}

\newtheorem{corollary}[thm]{Corollary}

\newtheorem{dfn}[thm]{Definition}
\newcommand{\C}[1]{\boldsymbol{\mathcal{#1}}}
\newcommand{\ud}{\mathrm{d}}
\newcommand{\ds}{\displaystyle}

\newcommand{\dlb}{\llbracket}
\newcommand{\drb}{\rrbracket}
\newcommand{\ov}[1]{\overline{#1}}
\newcommand{\B}[1]{\mathds{#1}}
\newcommand{\oB}[1]{\ov{\mathds{#1}}}
\newcommand{\oC}[1]{\ov{\C{#1}}}
\newcommand{\wh}[1]{\widehat{#1}}
\newcommand{\wt}[1]{\widetilde{#1}}
\newcommand{\myeq}[1]{{\rm (\ref{#1})} on page \pageref{#1}}
\DeclareMathOperator*{\ess}{ess}

\newcommand{\ts}{\textstyle}
\newcommand{\tint}{{\ts \int}}
\renewcommand{\epsilon}{\varepsilon}

\endlocaldefs

\begin{document}

\begin{frontmatter}
%%%%%%%%%%%%%%%%%%%%%%%%%%%%%%%%%%%%%%%%%%%%%%
%%                                          %%
%% Enter the title of your article here     %%
%%                                          %%
%%%%%%%%%%%%%%%%%%%%%%%%%%%%%%%%%%%%%%%%%%%%%%
\title{New bounds for $k$-means and information $k$-means}
%\title{A sample article title with some additional note\thanksref{T1}}
\runtitle{New bounds for $k$-means}
%\thankstext{T1}{A sample of additional note to the title.}

\begin{aug}
%%%%%%%%%%%%%%%%%%%%%%%%%%%%%%%%%%%%%%%%%%%%%%
%%Only one address is permitted per author. %%
%%Only division, organization and e-mail is %%
%%included in the address.                  %%
%%Additional information can be included in %%
%%the Acknowledgments section if necessary. %%
%%%%%%%%%%%%%%%%%%%%%%%%%%%%%%%%%%%%%%%%%%%%%%
\author[A]{\fnms{Gautier} \snm{Appert}\ead[label=e1]{gautier.appert.chess@gmail.com}}
\and
\author[B]{\fnms{Olivier} \snm{Catoni}\ead[label=e2]{olivier.catoni@ensae.fr}}
% \and
% \author[B]{\fnms{???} \snm{???}\ead[label=e3,mark]{???@???}}

%%%%%%%%%%%%%%%%%%%%%%%%%%%%%%%%%%%%%%%%%%%%%%
%% Addresses                                %%
%%%%%%%%%%%%%%%%%%%%%%%%%%%%%%%%%%%%%%%%%%%%%%
\address[A]{SAMM, Universit\'{e} Paris 1 Panth\'{e}on-Sorbonne, France, \printead{e1}}

\address[B]{CNRS -- CREST, UMR 9194, Universit\'{e} Paris Saclay, France, \printead{e2}}
\end{aug}

\begin{abstract}
In this paper, we derive a new dimension-free non-asymptotic upper bound 
for the quadratic $k$-means excess risk 
 related to the quantization of an i.i.d sample in a separable Hilbert space.
We improve the bound of order $\C{O} \bigl( k / \sqrt{n} \bigr)$ of 
Biau, Devroye and Lugosi, recovering the rate $\sqrt{k/n}$ 
that has already been proved by Fefferman,
Mitter, and Narayanan and by Klochkov,
 Kroshnin and Zhivotovskiy 
but with worse log factors and constants.
More precisely, we bound the mean excess risk of an empirical 
minimizer by the explicit upper bound $16 B^2
\log(n/k) \sqrt{k \log(k) / n}$, in the bounded 
case when $\B{P}( \lVert X \rVert \leq B) = 1$.
This is essentially optimal 
up to logarithmic factors since a lower bound 
of order $\C{O} \bigl( \sqrt{k^{1 - 4/d}/n} \bigr)$ is known 
in dimension $d$.
Our technique of proof is based on the linearization of the $k$-means 
criterion through a kernel trick and on PAC-Bayesian inequalities. 
To get a $1 / \sqrt{n}$ speed, we introduce a new PAC-Bayesian chaining 
method replacing the concept of $\delta$-net with 
the perturbation of the parameter by an infinite dimensional
Gaussian process.

In the meantime, 
we embed the usual $k$-means criterion into a broader family
built upon the Kullback divergence and its underlying 
properties. This results in a new algorithm that we named 
\textit{information $k$-means}, well suited to the clustering of bags 
of words. 
Based on considerations from information theory, 
we also introduce a new bounded $k$-means criterion
that uses a scale parameter but satisfies a 
generalization bound that does not require any boundedness
or even integrability conditions on the sample.
We describe the counterpart of Lloyd's algorithm 
and prove generalization bounds for these new 
$k$-means criteria. 
\end{abstract}



\end{frontmatter}
%%%%%%%%%%%%%%%%%%%%%%%%%%%%%%%%%%%%%%%%%%%%%%
%% Please use \tableofcontents for articles %%
%% with 50 pages and more                   %%
%%%%%%%%%%%%%%%%%%%%%%%%%%%%%%%%%%%%%%%%%%%%%%
%\tableofcontents

%%%%%%%%%%%%%%%%%%%%%%%%%%%%%%%%%%%%%%%%%%%%%%
%%%% Main text entry area:
\section*{General notation}
\addcontentsline{toc}{section}{General notation}

We will use the following notation throughout this document.

On some measurable probability space $\Omega$, we will consider various
random variables $X : \Omega \rightarrow \C{X}$, 
$Y:\Omega \rightarrow \C{Y}$, etc. that are nothing but measurable functions.
We will also consider several probability measures on $\Omega$, and typically 
two measures $\B{P}$ and $Q \in \C{M}_+^1(\Omega)$, where $\B{P}$ 
describes the usually unknown data distribution and $Q$ describes 
an estimation of $\B{P}$. Then we will use the short notation $\B{P}_X$
for the push forward measure $\B{P} \circ X^{-1}$, that is the 
law of $X$. Similarly we will 
let $Q_X = Q \circ X^{-1}$. In the same way $\B{P}_{X, Y} \in 
\C{M}_+^1( \C{X} \times \C{Y})$ will 
be the joint distribution of the couple $(X, Y)$ under $\B{P}$
and $\B{P}_{Y \, | \, X}$ the corresponding regular conditional 
probability measure of $Y$ knowing $X$ when it exists.
We will always work under sufficient hypotheses to ensure that
the decomposition 
\begin{equation}
\label{eq:condProba}
\B{P}_{X, \, Y} = \B{P}_X \B{P}_{Y \, | \, X} 
\end{equation}
is valid, meaning that for any bounded measurable function 
$f(X, Y)$
\[
\int f \, \ud \B{P}_{X, \, Y} = \int 
\biggl( \int f \ud \B{P}_{Y \, | \, X} \biggr) \, \ud \B{P}_X.
\] 
Moreover, we will use the short notation
\[
\int f \, \ud \B{P}_{X, \, Y} = \B{P}_{X, \, Y} ( f ),
\] 
so that the previous formula becomes
\[
\B{P}_{X, \, Y} (f) = \B{P}_X \bigl[ \B{P}_{Y \, | \, X} (f) \bigr].
\] 
We will often use the Kullback Leibler divergence
\[
\C{K} \bigl( Q, \B{P} \bigr) = \begin{cases}
\ds Q \biggl[ \log \biggl( \frac{\ud Q}{\ud \B{P}} \biggr) \biggr] 
& \text{ when } Q \ll \B{P},\\
+ \infty & \text{ otherwise.}
\end{cases}
\] 
We will always be in this article in a situation where the decomposition
\begin{lemma}
\label{lem:1.2}
\begin{align*}
\C{K} \bigl( Q_{X,Y}, \B{P}_{X,Y} \bigr)
&=\C{K}\bigl(Q_{X}, \B{P}_{X} \bigr) + Q_{X} \bigl[
\C{K}\bigl(Q_{Y\, | \,X}, \B{P}_{Y\, | \,X}
\bigr)\bigr]\\
&=\C{K}\bigl(Q_{Y}, \B{P}_{Y} \bigr)+Q_{Y} \bigl[\C{K}\bigl(Q_{X\, | \,Y},\B{P}_{X\, | \,Y}
\bigr)\bigr]
\end{align*}
\end{lemma}
\noindent is valid.
\begin{proof}
It follows from the decomposition \eqref{eq:condProba}. 
A precise statement and a rigorous proof dealing with measurability issues
can be found in
\cite[Appendix section 1.7 page 50]{catoni2}.
\end{proof}

\section{Introduction}

This paper is about the most widely used 
loss function for vector quantization, 
the $k$-means criterion.
We will be interested in the statistical setting 
where the problem is to minimize the criterion
for a random vector whose distribution is unknown
but can be estimated through an i.i.d. random 
sample. Our main contribution will be 
to prove a new dimension free non-asymptotic
generalization bound with a better $k/n$
dependence, where $k$ is the number of centers
used for vector quantization and $n$ is 
the size of the statistical sample.
We will also give an interpretation 
of the $k$-means criterion in terms of 
the Kullback divergence and use it
to embed it in 
a broader family of criteria with interesting 
properties. This will provide a specific 
algorithm for the quantization of conditional 
probability distributions ranging in an 
exponential family that can be used
in particular to analyse bag of words models. 
This generalization will also provide a 
new robust criterion for the quantization
of unbounded random vectors.

Our general setting is the following. Given a random 
variable $X \in H$ ranging in a separable Hilbert 
space $H$, we are interested in minimizing the risk 
function
\[
\C{R} \bigl( c_1, \dots, c_k \bigr) =  
\B{P}_X \Bigl( \min_{j \in \dlb 1, k \drb} \lVert X - c_j \rVert^2 
\Bigr), \quad (c_1, \dots, c_k) \in H^k.
\] 
We will assume that the statistician does not know the 
distribution $\B{P}_X$, but has access instead to 
a sample $(X_1, \dots, X_n)$ made of $n$ independent
copies of $X$. If $\ds \oB{P}_X = \frac{1}{n} \sum_{i=1}^n \delta_{X_i}$
is the empirical measure of the sample, 
the empirical risk, or empirical $k$-means criterion, is defined as
\[
\oC{R}(c) =  
\oB{P}_X \Bigl( \min_{j \in \dlb 1, k \drb} \lVert X - c_j \rVert^2 \Bigr) 
= \frac{1}{n} \sum_{i=1}^n \Bigl( \min_{j \in \dlb 1, k \drb} \lVert X_i - c_j \rVert^2 \Bigr), 
\quad c \in H^k.
\]
We will first consider the bounded case.
Given a ball $\C{B} = \bigl\{ x \in H \, : \, \lVert x \rVert \leq B \bigr\}$,
we will assume that $\B{P} \bigl( X \in \C{B} \bigr) = 1$. We will
study the upper deviations of the random variable 
\[
\sup_{c \in \C{B}^k} \Bigl[ \C{R}(c) - \oC{R}(c) \Bigr].
\]
This will provide an observable upper bound for the risk that is 
uniform with respect to the choice of centers $c \in \C{B}^k$. 
Being uniform with respect to $c$ covers
the case where the centers have been computed from 
the observed sample through some algorithm.
In order to study the excess risk of estimators, 
we will also study the upper deviations of 
\[
\sup_{c \in \C{B}^k} \Bigl[ \C{R}(c) - \C{R}(c_*) - \oC{R}(c) 
+ \oC{R}(c_*) \Bigr].
\]
This random variable compares uniformly with respect 
to $c$ the excess risk of $c$ with respect to a 
non random reference $c_*$ and the corresponding 
empirical excess risk.
In particular, the non random reference $c_*$
can be chosen to be a minimizer, or more generally 
an $\epsilon$-minimizer, of the risk $\C{R}$.

Indeed, we will consider $\wh{c} \in \C{B}^k$,
depending on the sample, such that
\[
\oC{R}(\wh{c}\,) \leq \inf_{c \in \C{B}^k} \oC{R}(c) + \epsilon,
\]
and provide a bound for the excess risk $\C{R} (
\wh{c}\, ) - \inf_{c \in \C{B}^k} \C{R}(c)$. 
To complement deviation bounds, we will also provide 
corresponding bounds in expectation.
\iffalse
These bounds will be dimension free and of order 
$\C{O} \Bigl( \log(n/k) \sqrt{ k \log(k) / n} \Bigr)$,
improving on the existing literature.
They will be obtained using PAC-Bayesian inequalities 
and a new chaining method. 
\fi

%mianmian3
Regarding the sample size $n$,
we obtain a speed of order $ \C{O} \bigl( 1 / \sqrt{n} \bigr)$ as in
\cite{biau}, \cite{Fischer} and \cite{brecheteau}.
However, we get a better dependence in $k$, 
with a rate of convergence of $\C{O} \bigl( \sqrt{k/n} \bigr)$ up to log factors.
This is essentially optimal up to log
factors, at least in infinite dimension, 
since minimax lower bounds for the
excess $k$-means risk are of order $\C{O} \bigl( 
\sqrt{k^{1 - 4/d}/n} \bigr)$ in dimension $d$, see \cite{bartlett1998minimax} and \cite{antos2005improved}.
We should mention that
the speed $\sqrt{k/n}$ has already been established 
in \cite{fefferman2016testing} (see Lemma 6), 
\cite{foster2019} and \cite{Nikita}, but with worse
log factors and less explicit constants. 

% mianmian
These bounds will be obtained using PAC-Bayesian inequalities
combined with a kernel trick
and a new kind of PAC-Bayesian chaining method that we developed.
In particular, borrowing ideas
from the construction of the isonormal Gaussian process
\cite[section 3.5]{Massart2007ConcentrationIA},
we will use the distribution of
an infinite sequence of shifted Gaussian random variables
both for the prior and the posterior parameter distribution.
We will also use some arguments from the proofs of
\cite{CatGiu} and \cite{CatOlGiu},
concerning the estimation of the mean of a random vector.
Furthermore,
we take inspiration from the classical chaining procedure for bounding the expected suprema of
sub-Gaussian processes (see section 13.1 in \cite{boucheron}). We create a PAC-Bayesian version of chaining in which
the concept of $\delta$-net and $\delta$-covering is replaced by the use of a sequence of
Gaussian perturbations parametrized by a variance ranging on a
logarithmic grid.
We combine this PAC-Bayesian chaining with the use of the influence function $\psi$ described in
\cite{catoni2012} to decompose the excess risk into
a sub-Gaussian part and an other part representing extreme values.
It is worth mentioning that
we will work with weak hypotheses and will in particular
not consider the kind of margin assumptions that are
necessary to get bounds decreasing faster than $\sqrt{1/n}$
%for a given value of $k$, as in \cite{biau}, \cite{Fischer}, \cite{levrard},\cite{levrard2015} and \cite{brecheteau}.
% mianmian3
for a given value of $k$, as in  \cite{levrard2013fast}, \cite{levrardThesis},\cite{levrard2015}, and \cite{levrard2018}.

\section{Extensions of the $k$-means criterion}

Before proving generalization bounds, let us 
embed the $k$-means criterion in a broader family of 
risk functions.

We will do this while considering the square of the Euclidean
distance, or more generally in possibly infinite dimension 
the square of the Hilbert norm,
as the Kullback divergence between two Gaussian 
measures. In this interpretation, vector quantization 
according to the $k$-means criterion will appear 
as a special case of conditional probability measure
quantization according to an entropy criterion.

To describe things at a more technical level, we need
first to define the classification function underlying
vector quantization.

To a set of centers $c \in H^k$ indeed corresponds a classification
function $\ell : H \rightarrow \dlb 1, k \drb$ into Vorono\"i cells
defined as
\[
\ell(x) = \arg \min_{j \in \dlb 1, k \drb} \lVert x - c_j \rVert.
\] 
This definition may not be unique if the minimum is reached more 
than once, in which case, we make an arbitrary choice,
as for instance
\[
\ell(x) = \min \Bigl\{ 
\arg \min_{j \in \dlb 1, k \drb} \lVert x - c_j \rVert \Bigr\}.
\]
The corresponding vector quantization function is 
\[
f(x) = c_{\ell(x)}, \quad x \in H.
\] 

To study the quality of the quantization of a random 
variable $X \in H$ in terms of conditional probability distributions, 
we introduce another random 
variable $Y \in \B{R}^{\B{N}}$ and consider on 
some probability space $\Omega$ a realization $(X, Y) : \Omega \rightarrow 
H \times \B{R}^{\B{N}}$ of the couple of random variables $(X, Y)$.
We can for instance take $\Omega = H \times \B{R}^{\B{N}}$ 
and let $(X,Y)$ be the identity. Introduce now a probability measure
$\B{P} \in \C{M}_+^1(\Omega)$ such that $\B{P}_X$ is the 
law of $X$ and such that $\B{P}_{Y \, | \, X}$ is the law 
of the independent sequence $\langle X, e_i \rangle 
+ \sigma \epsilon_i$, $i \in \B{N}$,
where $\bigl( \epsilon_i, i \in \B{N} \bigr)$ 
is an i.i.d. sequence of standard normal
random variables, where $\bigl( e_i, i \in \B{N} \bigr)$
is a basis of $H$ and where $\sigma > 0$ is a standard deviation 
parameter. In other words, let 
\[
\B{P}_{Y \, | \, X} = \bigotimes_{i \in \B{N}} \C{N} \bigl( \langle X, e_i 
\rangle, \sigma^2 \bigr).
\] 
Let us also introduce a probability measure $Q^{(c)} \in \C{M}_+^1(\Omega)$
such that $Q^{(c)}_X = \B{P}_X$ and 
\[
Q^{(c)}_{Y \, | \, X} = \bigotimes_{i \in \B{N}} \C{N} \bigl( 
\langle c_{\ell(X)}, e_i \rangle, \sigma^2 \bigr),
\]
where $\ell$ is defined from $c$ as explained above.
In other words, $Q^{(c)}_{Y \, | \, X}$
is the distribution of the random sequence
$\langle c_{\ell(X)}, e_i \rangle + \sigma \epsilon_i, i \in \B{N}$.
We see that $Q^{(c)}_{Y \, | \, X}$ is a quantization of $\B{P}_{Y \, | \, X}$
that takes $k$ values, in the same way as $f(X) = c_{\ell(X)}$ 
is a quantization of $X$ itself. 
\begin{prop}
\label{prop:02}
The $k$-means criterion
$\C{R}$ can be expressed as
\begin{equation}
\label{eq:02}
\C{R} (c) = 2 \sigma^2 \B{P}_X \bigl[ \C{K} \bigl( Q^{(c)}_{Y \, | \, X}, 
\B{P}_{Y \, | \, X} \bigr) \bigr] = 2 \sigma^2 \C{K} \bigl( 
Q^{(c)}_{X, \, Y}, \B{P}_{X, \, Y} \bigr).
\end{equation}
\end{prop}
\begin{proof}
The first equality comes from the fact that
\begin{multline*}
\C{K} \bigl( Q^{(c)}_{Y \, | \, X}, \B{P}_{Y \, | \, X} \bigr) 
= \sum_{i \in \B{N}} \C{K} \bigl( \C{N} ( \langle c_{\ell(X)}, e_i \rangle, 
\sigma^2 ), \C{N} ( \langle X, e_i \rangle, \sigma^2) \bigr)
\\ = \sum_{i \in \B{N}} \frac{1}{2 \sigma^2} \bigl( \langle c_{\ell(X)}, 
e_i \rangle - \langle X, e_i \rangle \bigr)^2
= \frac{1}{2 \sigma^2} \lVert X - c_{\ell(X)}\rVert^2
= \frac{1}{2 \sigma^2} \min_{j \in \dlb 1, k \drb} \lVert X - c_j \rVert^2.
\end{multline*}
The second equality is a consequence of the decomposition stated 
in Lemma \vref{lem:1.2}, that says that
\[
\C{K} \bigl( Q_{X, \, Y}^{(c)}, \B{P}_{X, \, Y} \bigr) = 
\C{K} \bigl( Q^{(c)}_X, \B{P}_X \bigr) + Q^{(c)}_X \bigl[ 
\C{K} \bigl( Q^{(c)}_{Y \, | \, X}, \B{P}_{Y \, | \, X} \bigr) 
\bigr]
\] 
and of the fact that $Q^{(c)}_X = \B{P}_X$ by definition of $Q^{(c)}$.
\end{proof}

The two equalities of Proposition \ref{prop:02} 
will be interesting to extend the $k$-means 
criterion. Let us draw the consequences of the first one first.
Introduce
\[
\mu^{(c)}_j = \bigotimes_{i \in \B{N}} \C{N} \bigl( \langle c_j, e_i \rangle, 
\sigma^2 \bigr) \in \C{M}_+^1(\B{R}^{\B{N}}), \qquad j \in \dlb 1, k \drb. 
\]
The first part of equation \eqref{eq:02} can be written as
\[
\C{R}(c) = 2 \sigma^2 \B{P}_{X} \bigl[ \C{K} \bigl( 
\mu_{\ell(X)}^{(c)}, \B{P}_{Y \, | \, X} \bigr) \bigr]
= 2 \sigma^2 \B{P}_{X} \bigl[ \min_{j \in \dlb 1, k \drb} 
\C{K} \bigl( \mu^{(c)}_j, \B{P}_{Y \, | \, X} \bigr) \bigr],
\] 
since 
\[
2 \sigma^2 \, \C{K} \bigl( \mu^{(c)}_j, \B{P}_{Y \, | \, X} \bigr) 
= \lVert X - c_j \rVert^2.
\]
Note that we could also have used 
$2 \sigma^2 \, \C{K} \bigl( \B{P}_{Y \, | \, X} , \mu^{(c)}_j \bigr) 
= \lVert X - c_j \rVert^2$, since the Kullback divergence between
Gaussian measures is symmetric. This would have lead to another 
interpretation of the $k$-means criterion in a space of conditional
probability measures. The choice we made is quite unusual, but is 
justified by the following property.
\begin{prop}
\label{prop:03}
The minimization of $\C{R}(c)$ seen as a function of $\mu^{(c)}$
can be extended to the larger set $\C{M}_+^1 \bigl( \B{R}^{\B{N}} \bigr)^k$.
In other words, if we put
\[
\wt{\C{R}} ( \mu ) = 
2 \sigma^2 
\B{P}_X \bigl[ \min_{j \in \dlb 1, k \drb} 
\C{K} \bigl( \mu_{j} , \B{P}_{Y \, | \, X} \bigr) \bigr], \qquad 
\mu \in  \C{M}_+^1( \B{R}^{\B{N}})^k, 
\]
we can see that $\C{R}(c) = \wt{\C{R}}(\mu^{(c)})$
and the minimum of $\wt{\C{R}}$ coincides with 
the minimum of $\C{R}$ in the sense that 
\[ 
\inf_{c \in H^k} \C{R}(c) 
= \inf_{c \in H^k} \wt{\C{R}}(\mu^{(c)}) 
= \inf_{ 
\mu \in  \C{M}_+^1( \B{R}^{\B{N}})^k 
} \wt{\C{R}}(\mu). 
\] 
\end{prop}
\begin{proof}
Given $\mu \in \C{M}_+^1( \B{R}^{\B{N}} )^k$,
we have to find $c \in H^k$ such that
$\wt{\C{R}}(\mu) \geq \C{R}(c)$.
This will prove that
\[ 
\inf_{c \in H^k} \C{R}(c) 
= \inf_{c \in H^k} \wt{\C{R}}(\mu^{(c)}) 
\leq \inf_{ 
\mu \in  \C{M}_+^1( \B{R}^{\B{N}})^k 
} \wt{\C{R}}(\mu), 
\] 
and since the reverse inequality is obvious
from the fact that we take the infimum on 
a larger set, this will prove the 
proposition.

Consider then 
\[
\ell(X) = \min \Bigl\{ \arg \min_{j \in \dlb 1, k \drb} 
\C{K} \bigl( \mu_j, \B{P}_{Y \, | \, X} \bigr) \Bigr\}
\]
and
\[
c_j = \B{P}_{X \, | \, \ell(X) = j} ( X ), \qquad j \in \dlb 1, k \drb.
\]
It is easy to check that the centers $c_j$ are such that
\[
\frac{\ud \mu^{(c)}_j}{\ud \mu^{(0)}_j} = Z_j^{-1} \exp \biggl\{
\B{P}_{X \, | \, 
\ell(X) = j} \biggl[ 
\log \biggl( \frac{\ud \B{P}_{Y \, | \, X}}{\ud \mu_j^{(0)}}
\biggr) \biggr] \biggr\},
\]
where $Z_j$ is a normalizing constant and $\mu^{(0)}$ is $\mu^{(c)}$ 
with $c = 0 \in H^k$, the centered Gaussian measure. Indeed, Gaussian measures
with the same covariance form an exponential family indexed by 
their means. Taking the arithmetic mean of the parameter in 
an exponential family results in taking the geometric mean
of the probability measures.
Thus, $\mu_j^{(c)}$ is the geometric mean of 
$\B{P}_{Y \, | \, X}$ with weights $\B{P}_{X \, | \, \ell(X) = j}$.
As a consequence, for any $j \in \dlb 1, k \drb$,  
\begin{multline*}
\B{P}_{X \, | \, \ell(X) = j} \bigl[ \C{K} \bigl( 
\mu_{\ell(X)}, \B{P}_{Y \, | \, X} \bigr) \bigr] \\ =
\B{P}_{X \, | \, \ell(X) = j} \biggl\{ 
\mu_j \biggl[ \log \biggl( \frac{ \ud \mu_j }{ 
\ud \B{P}_{Y \, | \, X}} \biggr) \biggr] \biggr\} \\
= 
\mu_j \biggl[ \log \biggl( \frac{\ud \mu_j}{\ud \mu_j^{(0)}} 
\biggr) \biggr] - \mu_j \biggl\{ \B{P}_{X \, | \, \ell(X) = j} 
 \biggl[ \log \biggl( \frac{\ud \B{P}_{Y \, | \, X}}{\ud \mu^{(0)}_j} 
\biggr) \biggr] \biggr\} \\ = 
\mu_j \biggl[ \log \biggl( \frac{\ud \mu_j}{\ud \mu_j^{(0)}} \biggl) 
\biggr] - \mu_j \biggl[ \log \biggl( \frac{\ud \mu_j^{(c)}}{ \ud
\mu_j^{(0)}} \biggr) \biggr] - \log(Z_j)
= \mu_j \biggl[ \log \biggl( \frac{\ud \mu_j}{\ud \mu_j^{(c)}} 
\biggr) \biggr] - \log(Z_j).  
\end{multline*}
Moreover, considering the case when $\mu = \mu^{(c)}$, 
we see that
\[ 
\B{P}_{X \, | \, \ell(X) = j} \bigl[ \C{K} \bigl( 
\mu_{\ell(X)}^{(c)}, \B{P}_{Y \, | \, X} \bigr) \bigr] = - \log(Z_j).
\]
Therefore
\begin{multline*}
\B{P}_X \bigl[ \min_{j \in \dlb 1, k \drb} \C{K} \bigl( 
\mu_j, \B{P}_{Y \, | \, X} \bigr) \bigr] = 
\B{P}_X \bigl[ \C{K} \bigl( 
\mu_{\ell(X)}, \B{P}_{Y \, | \, X} \bigr) \bigr]
= \B{P}_{\ell(X)} \B{P}_{X \, | \, \ell(X)} \bigl[ \C{K} \bigl( 
\mu_{\ell(X)}, \B{P}_{Y \, | \, X} \bigr) \bigr]
\\ = \B{P}_X \bigl[ \C{K} \bigl( 
\mu_{\ell(X)}, \mu^{(c)}_{\ell(X)} \bigr) \bigr] + 
\B{P}_X \bigl[ \C{K} \bigl( \mu^{(c)}_{\ell(X)}, 
\B{P}_{Y \, | \, X} \bigr) \bigr]
\\ \geq 
\B{P}_X \bigl[ \C{K} \bigl( \mu^{(c)}_{\ell(X)}, 
\B{P}_{Y \, | \, X} \bigr) \bigr] \geq 
\B{P}_{X} \bigl[ \min_{j \in \dlb 1, k \drb} 
\C{K} \bigl( \mu^{(c)}_j, \B{P}_{Y \, | \, X} \bigr) \bigr],
\end{multline*}
showing that $\wt{\C{R}}(\mu) \geq \wt{\C{R}} \bigl( \mu^{(c)} \bigr)$.
\end{proof}

So Proposition \ref{prop:03} shows that the $k$-means
algorithm also solves a quantization problem for 
Gaussian conditional probability measures $\B{P}_{Y \,  | \, X}$.
This is an invitation to study more generally 
the quantization problem for conditional probability measures,
using what we will call 
the information $k$-means
criterion $\wt{\C{R}}(\mu)$. 
This is what will be done in section \vref{section:7}.

Let us now come back to the second equality of equation \myeq{eq:02}.
It relates the minimization of the $k$-means criterion 
with the estimation of the joint probability measure $\B{P}_{X, \, Y}$.
Instead of considering the single distribution $Q^{(c)}_{X, \, Y}$
we can optimize the value of $Q^{(c)}_X$, considering the model
\[
\C{Q}(c) = \bigl\{ Q \in \C{M}_+^1(\Omega) \, : \, 
Q_{Y \, | \, X} = \mu^{(c)}_{\ell(X)} \bigr\} \ni
Q^{(c)}_{X, \, Y}.
\]
In order to get a better approximation of $\B{P}_{X, \, Y}$, 
it is natural to consider instead of $\C{R}(c)$
the criterion 
\[
\C{C}_2(c) = 2 \sigma^2 \inf_{Q \in \C{Q}(c)} \C{K} \bigl( 
Q_{X, \, Y}, \B{P}_{X, \, Y} \bigr) \leq 
\C{R}(c) = 2 \sigma^2 \C{K} \bigl( Q^{(c)}_{X, \, Y}, \B{P}_{X, \, Y} 
\bigr).
\] 
It turns out that this infimum can be computed.
\begin{prop}
Consider the classification function
\begin{equation}
\label{eq:classif}
\ell_c(x) = \min \Bigl\{ \arg \min_{j \in \dlb 1, \, k \drb} 
\lVert X - c_j \rVert \Bigr\}.
\end{equation}
The above criterion is equal to
\begin{align*}
\C{C}_2(c) = 2 \sigma^2 \inf_{Q \in \C{Q}(c)} \C{K} \bigl( 
Q_{X, \, Y}, \B{P}_{X, \, Y} \bigr) & = 
- 2 \sigma^2 \log \B{P}_X \Bigl\{ \exp \bigl[ 
- \C{K} \bigl( \mu_{\ell_c(X)}^{(c)}, \B{P}_{Y \, | \, X} 
\bigr) \bigr] \Bigr\} \\ 
& = - 2 \sigma^2 \log \B{P}_X \Bigl\{ 
\exp \Bigl[ - \frac{1}{2 \sigma^2} 
\bigl\lVert X - c_{\ell_c(X)} \bigr\rVert^2 \Bigr] \Bigr\} \\
& = - 2 \sigma^2 \log \B{P}_X \Bigl\{ 
\exp \Bigl[ - \frac{1}{2 \sigma^2} 
\min_{j \in \dlb 1, k \drb} \bigl\lVert X - c_{j} \bigr\rVert^2 \Bigr] \Bigr\}.
\end{align*}
\end{prop}
\begin{proof}
For any $Q \in \C{Q}(c)$, use the decomposition stated 
in Lemma \vref{lem:1.2}, to obtain that
\begin{multline}
\label{eq:robustDecomp}
\C{K} \bigl( Q_{X, \, Y}, \B{P}_{X, \, Y} \bigr)
= \C{K} \bigl( Q_X, \B{P}_X \bigr) + Q_X \bigl[ \C{K} \bigl( 
Q_{Y \, | \, X}, \B{P}_{Y \, | \, X} \bigr) \bigr]
\\ = \C{K} \bigl( Q_X, \B{P}_X \bigr) + 
Q_X \bigl[ \C{K} \bigl( \mu^{(c)}_{\ell_c(X)}, \B{P}_{Y \, | \, X} 
\bigr) \bigr]. 
\end{multline}
Minimizing this last expression with respect to $Q_X \in \C{M}_+^1(H)$
according to forthcoming Lemma \vref{lem:1.3} gives the first 
equality of the proposition, the others being obvious.
\end{proof}

The criterion $\C{C}_2(c)$ is not a risk function in the sense
that it is not the expectation of a loss function, but it 
is closely related to one. Indeed we can introduce
\begin{equation}
\label{eq:robustQuadratic}
\C{R}_2(c) = 2 \sigma^2 \Bigl[ 1 - \exp \Bigl( - \frac{1}{2 \sigma^2} 
\C{C}_2(c) \Bigr) \Bigr] \leq \C{C}_2(c) \leq \C{R}(c) 
\end{equation}
that is equal to 
\[
\C{R}_2(c) = 2 \sigma^2 \B{P}_X \Bigl[ 1 - \exp \Bigl( - \frac{1}{2 
\sigma^2} \min_{j \in \dlb 1, k \drb} \lVert X - c_j \rVert^2 \Bigr) 
\Bigr] 
\]
according to the previous proposition.
We see that the risk $\C{R}_2$ is a natural
modification of the risk $\C{R}$ when we relate
$\C{R}$ to the estimation of $\B{P}_{X, \, Y}$.
This new risk $\C{R}_2$ is smaller, meaning that
it should be easier to minimize and indeed, as 
it is the expectation of a bounded loss function,
we will get a generalization bound under weaker 
hypotheses than what we will ask for $\C{R}$.
More specifically, we will assume no more 
that the sample is bounded.

\section{Study of the robust quadratic $k$-means criterion}

We can find a local minimum of the usual quadratic $k$-means criterion
using Lloyd's algorithm that updates the centers and the classification
function alternately. In this section, we will describe a similar
algorithm for the robust criterion of equation 
\eqref{eq:robustQuadratic}. According to this equation, $\C{R}_2(c)$
is an increasing function of $\C{C}_2(c)$, so that we can as well
study the minimization of $\C{C}_2(c)$. The discussion will also 
cover the minimization of the corresponding empirical criteria,
replacing the law of $X$, $\B{P}_{X}$, by the empirical measure
$\oB{P}_X$.

According to the decomposition \eqref{eq:robustDecomp}, 
\[ 
\frac{1}{2 \sigma^2} \C{C}_2(c) = \inf_{Q_X \in \C{M}_+^1(H)} 
\C{K} \bigl( Q_X, \B{P}_X \bigr) + Q_X \Bigl( \frac{1}{2 \sigma^2} 
\lVert X - c_{\ell_c(X)} \rVert^2 \Bigr). 
\] 
Moreover the infimum in $Q_X$ is reached at $Q^*_X \ll \B{P}_X$ 
defined by its density
\[
\frac{ \ud Q^*_{X}}{\ud \B{P}_X} = 
Z^{-1} \exp \Bigl( - \frac{1}{2 \sigma^2} \lVert X - c_{\ell_c(X)} 
\rVert^2 \Bigr).
\]
This proves
\begin{prop}[Lloyd's algorithm for the robust $k$-means criterion]
For any $c \in H^k$, consider the updated centers $c' \in H^k$
defined as 
\[ 
c'_j = Q^*_{X \, | \, \ell_c(X) = j}(X) = \frac{\ds 
\B{P}_{X \, | \, \ell_c(X) = j} \Bigl[ X \exp \Bigl( - \frac{1}{2 \sigma^2} 
\bigl\lVert X - c_j \bigr\rVert^2 \Bigr) \Bigr]}{\ds 
\B{P}_{X \, | \, \ell_c(X) = j} \Bigl[ \exp \Bigl( - \frac{1}{2 \sigma^2} 
\bigl\lVert X - c_j \bigr\rVert^2 \Bigr) \Bigr]},
\] 
where $\ell_c$ is defined by equation \myeq{eq:classif}.
Then
\[
\C{C}_2(c') \leq \C{C}_2(c) - Q^*_X \bigl( \lVert c_{\ell_c(X)} - c'_{\ell_c
(X)} \rVert^2 \bigr) \leq \C{C}_2(c).
\] 
Accordingly $\C{R}_2(c') \leq \C{R}_2(c)$.
\end{prop}
So the update of the classification function is the same as
in the usual case, and the update of the centers performs
a conditional mean with exponential weights instead of the 
conditional mean used in the original Lloyd's algorithm.

\begin{proof}
We can see that
\begin{align*}
\frac{1}{2 \sigma^2} \C{C}_2(c) & = 
\C{K} \bigl( Q^*_X, \B{P}_X \bigr) + Q^*_X \Bigl( \frac{1}{2 \sigma^2} 
\lVert X - c_{\ell_c(X)} \rVert^2 \Bigr) \\
& =  
\C{K} \bigl( Q^*_X, \B{P}_X \bigr) + Q^*_X \Bigl( \frac{1}{2 \sigma^2} 
\lVert X - c'_{\ell_c(X)} \rVert^2 \Bigr) + Q^*_X \Bigl( 
\frac{1}{2 \sigma^2} 
\lVert c_{\ell_c(X)} -  
c'_{\ell_c(X)} \rVert^2 \Bigr)  
\\ & \geq 
\C{K} \bigl( Q^*_X, \B{P}_X \bigr) + Q^*_X \Bigl( \frac{1}{2 \sigma^2} 
\lVert X - c'_{\ell_{c'}(X)} \rVert^2 \Bigr) + 
Q^*_X \Bigl( \frac{1}{2 \sigma^2} 
\lVert c_{\ell_c(X)} -  c'_{\ell_c(X)} \rVert^2 \Bigr)  
\\ & \geq \frac{1}{2 \sigma^2} \C{C}_2(c') + 
Q^*_X \Bigl( \frac{1}{2 \sigma^2} 
\lVert c_{\ell_c(X)} -  c'_{\ell_c(X)} \rVert^2 \Bigr),  
\end{align*}
keeping in mind that $Q^*_X$ depends on $c$.
\end{proof}

\section{Study of the information $k$-means criterion}
\label{section:7}

In this section, we will study the information $k$-means criterion
$\wt{R}(\mu)$ of Proposition \vref{prop:03} for more general
models of regular conditional probability measures $\B{P}_{Y \, | \, X}$.

Consider a couple of random variables $(X, Y) \in \C{X} \times \C{Y}$,
where $\C{X}$ and $\C{Y}$ are complete separable metric spaces, 
so that we can define regular conditional probability measures.
Suppose there exists a 
reference measure $\nu \in \C{M}_+^1 \bigl( \C{Y} \bigr)$ 
such that  $ \B{P} \Bigl( \B{P}_{Y \, | \, X} \ll \nu \Bigr) = 1$. 
Define $\ds p_X = \frac{\ud \B{P}_{Y \, | \, X}}{\ud \nu}$. 
We are interested in the case where $\B{P}_{Y \, | \, X}$ 
is known therefore providing  
a bag of words model. This means that each random sample $X$ is described 
by a random probability measure $\B{P}_{ Y \, | \, X}$. In the original 
bag of words model, $\C{Y}$ is a set of words, and $\B{P}_{Y \, | \, X}$ 
is the distribution of words in a text $X$ drawn at random from some 
corpus of texts. Here we include the case where $\C{X}$ and $\C{Y}$ 
can be more general measurable spaces.  

We introduce the following generalization 
of the criterion $\wt{\C{R}}$ of Proposition \vref{prop:03}, 
that we will name the
information $k$-means criterion: 
\[ 
\inf_{q \in \bigl( \B{L}^1_{+,1}(\nu) \bigr)^k} 
\B{P}_X \Bigl( \min_{j \in \dlb 1, k \drb} \C{K}(q_j, p_{X}) \Bigr), 
\] 
where $\dlb 1, k \drb = \{1, \dots, k\}$, $\B{L}^1_{+,1}(\nu) = 
\Bigl\{ q \in \B{L}^1(\nu) \, : \, q \geq 0, \int \! q \, \ud \nu = 1 \Bigr\}$
and 
\[ 
\C{K}(q_j, p_X) = 
\begin{cases} 
\int q_j \log \bigl( q_j / p_X \bigr) \, \ud \nu, 
& \int q_j \B{1} \bigl( p_X = 0 \bigr) \, \ud \nu = 0, \\ 
+ \infty, & \text{ otherwise} 
\end{cases}
\]
is the Kullback divergence between densities.
\iffalse
More generally the Kullback-Leibler divergence $\C{K}$ is defined as
$$
\C{K}(\rho,\pi) 
\overset{\mathrm{def}}{=}
\begin{cases}
 \int \log \left( \frac{ \text{d} \rho }{\text{d}\pi   }
 \right) \; \text{d}\rho, & \mathrm{if} \; \rho \ll \pi \;
 (\pi \; \mathrm{dominates} \; \rho)
 \\ +\infty, & \; \mathrm{otherwise.}
\end{cases}
$$
\fi
The purpose of this section is to discuss the general properties of the information $k$-means problem and to build a mathematical framework and algorithms to perform the minimization.
As we have seen in the previous section, we chose to study this algorithm
rather than the better known $k$-means divergence algorithm
\[ 
\inf_{q \in \bigl( \B{L}_{+,1}^1(\nu) \bigr)^k} \B{P}_X 
\Bigl( \min_{j \in \dlb 1, k \drb} \C{K} \bigl( p_X, q_j \bigr) \Bigr)
\] %coucou 
because of Proposition \vref{prop:03}, showing that our proposal contains
the classical Euclidean $k$-means as a special case. More generally, using
the divergence in the way we do when the conditional probability measures
$\B{P}_{Y \, | \, X}$ belong to an exponential family ensures that
the optimal centers for a given classification function $\ell$ 
belong to that same exponential family.%Olivier 

%mianmian
We should point out that clustering histograms or more generally probability distributions based on the Kullback divergence or other information criteria is not a new subject. It has been extensively used in text categorization and image indexing, especially in word clustering to extract features or reduce the original space dimension, see \cite{Pereira}, \cite{Tishby}, \cite{slonim1999agglomerative}, \cite{Dhillon}, \cite{cao2013}, \cite{wu},
and \cite{jiang}. The clustering is essentially performed using the aforementioned $k$-means divergence algorithm. However, in the information $k$-means framework we follow a different route since the grouping step is done by minimizing the Kullback divergence with respect to its first argument instead of its second one. This leads to very different centroids, computed as geometric means of distributions instead of arithmetic means, see \cite{ben1989entropic} and \cite{veldhuis2002centroid}.
This follows from the fact that the Kullback divergence is asymmetric. 
Nevertheless, symmetric extensions of the Kullback divergence built 
upon averaged symmetrizations have been studied. 
Especially, centroids and $k$-means type 
algorithms derived from symmetrized divergence functions
are analyzed in \cite{veldhuis2002centroid}, \cite{nielsen2013jeffreys}, \cite{nielsen2014clustering} and \cite{nielsen2019jensen}.

% mianmian3
Besides, 
following the set-up provided by the typical $k$-means divergence, \cite{Banerjee} presents a general 
$k$-means framework based on the Bregman divergence. The authors show that such criteria can be minimized iteratively using a $k$-means centroid-based algorithm. The Bregman distance encompasses many traditional similarity
measures such as the Euclidean distance, the Kullback divergence, 
the logistic loss and many others.
However, in the Kullback case, the minimization is 
performed with respect to the second
argument, and not the first as in our proposal.
Nevertheless, the study of a symmetrized version of the Bregman
divergence, and especially the resulting centroids coming from it,
is undertaken in \cite{nielsen2009sided}.

% mianmian
Our contribution in this paper is to provide 
a mathematical framework 
for the information $k$-means criterion.
In particular, we will prove generalization bounds
and deal with the infinite dimension case. 
%our setting will not be limited to the empirical scenario
%and finite dimension case, which have been the main set-up established so far.

Let us state some version of the Bayes rule that will 
be useful in the following discussion.
\begin{lemma}
\label{lem:1.1}
Let $\B{P}_{X,Y}$ be a joint distribution 
defined on the product of two Polish spaces.
The following statements are equivalent:
\begin{enumerate}
\item There exists a measure $\mu$ such that 
$\B{P}_{Y \, | \, X} \ll \mu$, $\B{P}_X$ almost surely;
\item  $\B{P}_{Y \, | \, X} \ll \B{P}_Y$, $\B{P}_X$ almost surely;
\item  $\B{P}_{X,Y} \ll \B{P}_X \otimes \B{P}_Y$;
\item  $\B{P}_{X \, | \, Y} \ll \B{P}_X$, $\B{P}_Y$ almost surely.
\end{enumerate}
Moreover, they imply the following identities between 
Radon–Nikodym derivatives:
\[
\frac{\ud\B{P}_{X,Y}}{\ud
\big( \B{P}_{X}\otimes \B{P}_{Y} \big)}=
\frac{\ud\B{P}_{Y\, | \,X}}{\ud\B{P}_{Y}}
=\frac{\ud\B{P}_{X\, | \,Y}}{\ud\B{P}_{X}}.
\]
\end{lemma}

\begin{proof}
To prove that {\it 1.} implies {\it 2.}, it is sufficient to show that
$\B{P}_{Y \, | \, X}\Bigl( \frac{\ud\B{P}_{Y}}{\ud\mu}=0 \Bigr)=0$, 
$\B{P}_X$ almost surely.
But when {\it 1.} is true
\[
 \B{P}_{Y \, | \, X}\Biggl( \frac{\ud\B{P}_{Y}}{\ud\mu}=0 \Biggr)=
 \int \B{1}\biggl( \frac{\ud\B{P}_{Y}}{\ud\mu}=0 \biggr)\, \frac{\ud\mathbb{P}_{Y\, | \,X}}{\ud\mu} \, \ud\mu.
\]
Thus by the Tonelli-Fubini theorem
\begin{align*}
\B{P}_{X}\Biggl(
\B{P}_{Y \, | \, X}\Biggl( \frac{\ud\B{P}_{Y}}{\ud\mu}=0 \Biggr)\Biggr)&=
\B{P}_{X}\Biggl(
\int \B{1}\biggl( \frac{\ud\B{P}_{Y}}{\ud\mu}=0 \biggr)\, \frac{\ud\mathbb{P}_{Y\, | \,X}}{\ud\mu} \, \ud\mu
\Biggr)\\&=\int \B{1}\biggl( \frac{\ud\B{P}_{Y}}{\ud\mu}=0 \biggr)\,\B{P}_X \left[ \frac{\ud\mathbb{P}_{Y\, | \,X}}{\ud\mu} \right] \, \ud\mu
\\&=\int \B{1}\biggl( \frac{\ud\B{P}_{Y}}{\ud\mu}=0 \biggr)\,\frac{\ud\B{P}_{Y}}{\ud\mu}\, \ud\mu=0.
\end{align*}
Therefore $\B{P}_{Y \, | \, X}\Bigl( \frac{\ud\B{P}_{Y}}{\ud\mu}=0 \Bigr)=0$, $\B{P}_X$ almost surely.
Obviously {\it 2.} implies {\it 1.} with $\mu = \B{P}_Y$. 
Now let us show that {\it 2.} 
implies {\it 3.}
Let $f$ be a bounded measurable function, we have by Fubini's theorem
\begin{align*}
\int f \, \ud\B{P}_{X,Y}&=
\int \biggl( \int f \, \ud\B{P}_{Y \, | \, X} \biggr)\, \ud\B{P}_{X}=
\int \biggl( \int f \, \frac{\ud\B{P}_{Y \, | \, X}}{\ud\B{P}_{Y}} \, \ud\B{P}_{Y} \biggr) \,\ud\B{P}_{X}\\&=
\int  f \, \frac{\ud\B{P}_{Y \, | \, X}}{\ud\B{P}_{Y}} \, \ud \bigl(\B{P}_{Y} \otimes \ud\B{P}_{X}\bigr),
\end{align*}
implying {\it 3.} and that $\B{P}_{X}$ almost surely
\[
\frac{\ud\B{P}_{Y\, | \,X}}{\ud\B{P}_{Y}}
=\frac{\ud\B{P}_{X,Y}}{\ud
\big( \B{P}_{X}\otimes \B{P}_{Y} \big)}.
\]
We will show now that {\it 3.} implies {\it 2.}
Let $f$ be a bounded measurable function, we have by Fubini's theorem
\begin{align*}
\int f \, \ud\B{P}_{X,Y}&=
\int  f \, \frac{\ud\B{P}_{X,Y}}{\ud \bigl(\B{P}_{X} \otimes \B{P}_{Y}\bigr)} \, \ud \bigl(\B{P}_{X} \otimes \ud\B{P}_{Y}\bigr)\\&=
\int \biggl( \int f \, \frac{\ud\B{P}_{X,Y}}{\ud \bigl(\B{P}_{X} \otimes \B{P}_{Y}\bigr)} \, \ud\B{P}_{Y} \biggr)\, \ud\B{P}_{X}\\
&=\int \biggl( \int f \, \ud\B{P}_{Y \, | \, X} \biggr) \,\ud\B{P}_{X}
,
\end{align*}
showing that $\B{P}_{X}$ almost surely
$\B{P}_{Y \, | \, X} \ll \B{P}_Y$ and 
\[
\frac{\ud\mathbb{P}_{Y\, | \,X}}{\ud\mathbb{P}_{Y}}
=
\frac{\ud\mathbb{P}_{X,Y}}{\ud
\big( \mathbb{P}_{X}\otimes \mathbb{P}_{Y} \big)}.
\]
The equivalence between {\it 3.} and {\it 4.} is immediate by
interchanging the roles of $X$ and $Y$.
\end{proof}

The following lemma will be useful to optimize the 
information $k$-means criterion and is related to 
the Donsker Varadhan representation.
\begin{lemma}
\label{lem:1.3}
Let $\pi \in \mathcal{M} _{+}^{1}(\Omega)$ be a 
probability measure on the measurable space $\Omega$.
Let $h : \Omega \rightarrow \B{R} \cup \{ + \infty \}$ be a measurable 
function such that 
\[
Z = \int \exp( - h) \, \ud \pi < \infty.
\]
Let $\pi_{\exp(-h)}$
be the probability measure whose density 
with respect to $\pi$ is proportional to $\exp(-h)$ 
so that
\[
\frac{\ud \pi_{\exp(-h)} }{\ud \pi} =  \frac{\exp(-h)}{Z}.
\]
The identity
\begin{align*}
 \inf_{\eta \in \B{Z}} \biggl( \C{K}(\rho,\pi)  
+ 
\int \max \{ h, \eta \} \, \ud \rho \biggr) &=- 
\log \Biggl( \int \exp (-h) \, \ud \pi
\Biggr) + \C{K}(\rho,\pi_{\exp(-h)}) \in \B{R} \cup \{ + \infty \} 
\end{align*}
is satisfied for any $\rho \in \C{M}_+^1(\Omega)$ and implies that
\[ 
\inf_{\rho \in \C{M}_+^1( \Omega )} \inf_{\eta \in \B{Z}} \Biggl( \C{K}(\rho, \pi) + 
 \int \max \{ h, \eta \} \, \ud \rho \Biggr)=- \log \Biggl( \int \exp ( -h) \, \ud \pi
\Biggr),
\] 
the minimum being reached when $\rho = \pi_{\exp(-h)}$.
\end{lemma}
Note that the lemma could also be written as
\[
\C{K}(\rho, \pi) + \int h \, \ud \rho = - \log \biggl( 
\int \exp( - h) \, \ud \pi \biggr) + \C{K} \bigl( \rho, \pi_{\exp(-h)} \bigr)
\] 
if we are willing to follow the convention that 
\[ 
\int h \, \ud \rho = \inf_{\eta \in \B{Z}} \int \max \{ h, \eta \} \, 
\ud \rho
\] 
and that $+ \infty - \infty = + \infty$.

\begin{proof}
See \cite[page 159]{catoni2}.
Note that the role of $\eta \in \B{Z}$ in this lemma is only 
to make sure that the integrals are always well defined 
in $\B{R} \cup \{ + \infty \}$ in the 
sense that the negative part of the integrand is integrable.
When $\rho$ is not absolutely continuous with respect to $\pi$, 
it is also not absolutely continuous with respect to $\pi_{\exp(-h)}$
since $\pi(A) = 0$ if and only if $\pi_{\exp(-h)}(A) = 0$.
In this case $\C{K}(\rho, \pi) = \C{K}(\rho, \pi_{\exp(-h)}) = + \infty$
and the identity is true, both sides being equal to $+ \infty$.
When $\rho \ll \pi$, then $\rho \ll \pi_{\exp(- 
\max \{ h, \eta \})}$ and 
\[
\frac{\ud \rho}{\ud \pi_{\exp(-\max\{h, \eta\})}} = Z_{\eta} 
\exp(\max \{ h, \eta \}) \frac{\ud \rho}{\ud \pi},
\]
where 
\[
Z_{\eta} = \int \exp( - \max \{ h, \eta \} ) \, \ud \pi < + \infty.
\]
Therefore
\[
\C{K} \bigl( \rho, \pi_{\exp ( - \max \{ h, \eta \} )} \bigr) = 
\log \bigl(Z_{\eta}\bigr) + \int \Bigl[ \max\{h,\eta\} + \log \Bigl( 
\frac{\ud \rho}{\ud \pi} \Bigr) \Bigr] \, \ud \rho.
\] 
By the monotone convergence theorem 
\[
\lim_{\eta \rightarrow - \infty} Z_{\eta} = Z \text{ and } 
\lim_{\eta \rightarrow - \infty}  
\int \Bigl[ \max\{h,\eta\} + \log \Bigl( 
\frac{\ud \rho}{\ud \pi} \Bigr) \Bigr] \, \ud \rho
= \int \Bigl[ h + \log \Bigl( 
\frac{\ud \rho}{\ud \pi} \Bigr) \Bigr] \, \ud \rho,
\]
since we know that
\[ 
\int \Bigl[ \log(Z) + h + \log \Bigl( 
\frac{\ud \rho}{\ud \pi} \Bigr) \Bigr]_- \, \ud \rho
= \int \log \Bigl( \frac{\ud \rho}{\ud \pi_{\exp( - h)}} \Bigr)_- 
\;
\frac{\ud \rho}{\ud \pi_{\exp(-h)}} \,
\ud \pi_{\exp(-h)} \leq \exp( - 1 )  < + \infty
\] 
and therefore that
\[ 
\int \Bigl[ h + \log \Bigl( 
\frac{\ud \rho}{\ud \pi} \Bigr) \Bigr]_- \, \ud \rho < + \infty.
\] 
This proves that
\begin{multline*}
\lim_{\eta \rightarrow - \infty} 
\C{K} \bigl( \rho, \pi_{\exp ( - \max \{ h, \eta \} )} \bigr) = 
\log(Z) + \int \Bigl[ h + \log \Bigl( \frac{\ud \rho}{\ud \pi} \Bigr) 
\Bigr] \, \ud \rho = 
\C{K} \bigl( \rho, \pi_{\exp ( - h )} \bigr) \\ = 
\log(Z) + \inf_{\eta \in \B{Z}} \int \Bigl[ \max \{ h, \eta \} 
+ \log \Bigl( \frac{\ud \rho}{\ud \pi} \Bigr) \Bigr] \, \ud \rho
\\ = \log(Z) + \inf_{\eta \in \B{Z}} \biggl( \int \max \{ h, \eta \} \, 
\ud \rho + \int \log \Bigl( \frac{\ud \rho}{\ud \pi} \Bigr) \, \ud \rho
\biggr) \\
= \log(Z) + \inf_{\eta \in \B{Z}} \biggl( 
\int \max \{ h, \eta \} \, \ud \rho + \C{K}(\rho, \pi) \biggr),
\end{multline*}
and therefore that
\[
\C{K} \bigl( \rho, \pi_{\exp(-h)} \bigr) - \log(Z) = 
\inf_{\eta \in \B{Z}} \biggl( \C{K}(\rho, \pi) + \int
\max \{ h, \eta \} \, \ud \rho \biggr)
\] 
as stated in the lemma. The second statement of the lemma
is a consequence of the fact that the Kullback divergence 
is non negative.
\end{proof}

Let us now formulate a precise definition of the 
geometric mean of conditional probability measures
and show that it is their optimal center according 
to the information projection criterion.
\begin{lemma}
\label{lem:1.4}
Let $\B{P}_{X,Y}$ be a joint distribution defined on the 
product of two Polish spaces.
Assume 
that $\B{P}_X \bigl( \B{P}_{Y \, | \, X} \ll \B{P}_Y \bigr) = 1$. 
Consider the normalizing constant
\[
Z= \B{P}_{Y}\biggl( \exp \bigl[ -\C{K} \bigl( \B{P}_{X}, \, \B{P}_{X \, | \, Y} \bigr)\bigr] \biggr)
= \B{P}_Y \Biggl( \exp \biggl\{ \B{P}_X \biggl[ \log \biggl( 
\frac{ \ud \B{P}_{Y \, | \, X}}{\ud \B{P}_Y} \biggr) \biggr] \biggr\} 
\Biggr).
\]
Obviously, $Z \in [0,1]$. 
If $Z = 0$, then
\[ 
\inf_{Q_Y \in \C{M}_+^1(\C{Y})} \,
 \B{P}_X \bigl[ \C{K}( Q_Y ,\B{P}_{Y \, | \, X} )
\bigr] = + \infty. 
\] 
Otherwise, $Z > 0$ and for any $Q_Y \in \C{M}_+^1(\C{Y})$, 
\[
\B{P}_X \bigl[ \C{K} \bigl( Q_Y , \B{P}_{Y \, | \, X} \bigr) \bigr] = 
\C{K} \bigl( Q_Y, Q^{\star}_Y \bigr) + \B{P}_X \bigl[ 
\C{K} \bigl( Q^\star_Y, \B{P}_{Y \, | \, X} \bigr) \bigr]
= \C{K} \bigl( Q_Y, Q^{\star}_Y \bigr) 
+ \log \bigl( Z^{-1} \bigr),
\]
where $Q^{\star}_Y \ll \B{P}_Y$ is defined by the relation
\begin{align}
\nonumber
 \frac{\ud Q_{Y}^{\star}}{\ud \B{P}_{Y}}
 &= Z^{-1} \, \exp \bigl[ -\C{K} \bigl( \B{P}_{X}, \, \B{P}_{X \, | \, Y} \bigr)\bigr]\\
\label{eq:geomMean}
 &= Z^{-1} \,\exp \Biggl\{ \B{P}_X \Biggl[ \log \Biggl( \frac{\ud \B{P}_{Y \, | \, X}}{\ud 
\B{P}_{Y}} \Biggr)  \Biggr]   \Biggr\}.
\end{align}
Consequently
\[
\inf_{Q_Y \in \C{M}_+^1(\C{Y})} \,
 \B{P}_X \bigl[ \C{K}( Q_Y ,\B{P}_{Y \, | \, X} )\bigr] =\B{P}_X \bigl[ \C{K}( Q_{Y}^{\star} ,\B{P}_{Y \, | \, X} )\bigr]
 =\log\bigl(Z^{-1}\bigr) < \infty,
\]
The probability measure $Q_{Y}^{\star}$ represents the geometric 
mean of $\B{P}_{Y \, | \, X}$ with respect to
$\B{P}_X$. 
\end{lemma} 
\begin{proof}
By Lemma \vref{lem:1.2},
\begin{equation}
\label{eq:IKmeans}
\B{P}_X \bigl[ \C{K} \bigl( Q_Y, \B{P}_{Y \, | \, X} \bigr) \bigr] 
= \C{K} \bigl( \B{P}_X \otimes Q_Y, \B{P}_{X, \, Y} \bigr) \bigr] 
= \C{K} \bigl( Q_Y, \B{P}_Y \bigr) + Q_Y \bigl[ \C{K} \bigl( 
\B{P}_X, \B{P}_{X \, | \, Y} \bigr) \bigr]. 
\end{equation}
Thus, when \eqref{eq:IKmeans} is finite, $Q_Y \ll \B{P}_Y$
and 
\[
Q_Y \Bigl[ \C{K} \bigl( \B{P}_X, \B{P}_{X \, | \, Y} \bigr) < + \infty \bigr) \Bigr] = 1, 
\] 
so that 
\[
\B{P}_Y \Bigl[ \C{K} \bigl( \B{P}_X, 
\B{P}_{X \, | \, Y} \bigr) < + \infty \bigr) \Bigr] > 0, 
\]
implying that $Z > 0$. 
Assuming from now on that \eqref{eq:IKmeans} is finite, introduce 
\[
\C{A} = \Bigl\{ y \, : \, 
\C{K} \bigl( \B{P}_X, 
\B{P}_{X \, | \, Y = y} \bigr) < + \infty \Bigr\}. 
\]
From Lemma \vref{lem:1.3} and \eqref{eq:IKmeans}, for any $Q_Y
\in \C{M}_+^1(\C{A})$,
\begin{multline*}
\B{P}_X \bigl[ \C{K} \bigl( Q_Y, \B{P}_{Y \, | \, X} \bigr) \bigr] 
\\ = \underbrace{- \log \bigl[ \B{P}_Y(\C{A}) \bigr] 
- \log \B{P}_{Y \, | \, Y \in \C{A}} \Bigl\{ \exp \Bigl[ 
- \C{K} \bigl( \B{P}_X, \B{P}_{X \, | \, Y} \bigr) \Bigr] \Bigr\}}_{
= \log \bigl( Z^{-1} \bigr)} + 
\C{K} \bigl( Q_Y, Q^{\star}_Y \bigr) \\ 
= \B{P}_X \bigl[ \C{K} \bigl( Q^\star_Y, \B{P}_{Y \, | \, X} \bigr) \bigr] 
+ 
\C{K} \bigl( Q_Y, Q^{\star}_Y \bigr). 
\end{multline*}
Moreover, when $Q_Y(\C{A}) < 1$, $Q_Y \not\ll Q^\star_Y$, 
so that both members are equal to $ + \infty$.
The identity \eqref{eq:geomMean} is a consequence of Lemma 
\vref{lem:1.1}.
\end{proof}

We are now ready to express the minimum of the information 
$k$-means criterion in different ways involving the underlying
classification function and optimal centers.
\begin{prop}
\label{prop:1.5}
The information $k$-means problem can be expressed as
\begin{multline*}\inf_{q \in \bigl( \B{L}^1_{+,1}(\nu) \bigr)^k} \B{P}_X \Bigl( \min_{j \in \dlb 1, k \drb} \C{K}(q_j, p_{X}) \Bigr)
=\inf_{\ell:\C{X}\mapsto \dlb 1, k \drb} \;
\inf_{(q_1,\dots,q_k) \in \bigl( \B{L}^1_{+,1}(\nu) \bigr)^k} \;
\B{P}_X \Bigl(  \C{K}(q_{\ell(X)}, p_{X}) \Bigr) \\
 =\inf_{(q_1,\dots,q_k) \in \bigl( \B{L}^1_{+,1}(\nu) \bigr)^k} \;
\inf_{\ell:\C{X}\mapsto \dlb 1, k \drb} \;
\; \B{P}_X \Bigl(  \C{K}(q_{\ell(X)}, p_{X}) \Bigr) 
\\ =\inf_{(q_1,\dots,q_k) \in \bigl( \B{L}^1_{+,1}(\nu) \bigr)^k}
\; \B{P}_X \Bigl(  \C{K}(q_{\ell^{\star}_q(X)}, p_{X}) \Bigr) \\
 = \inf_{\ell:\C{X}\mapsto \dlb 1, k \drb} \, \B{P}_X \Bigl(  
\C{K}(q^{\star,\ell}_{\ell(X)}, p_{X}) \Bigr)
\\ = \inf_{ \ell:\C{X}\mapsto \dlb 1, k \drb}\;  \B{P}_X \Bigl( \log \bigl( Z^{-1}_{\ell(X)} \bigr)
\Bigr),
\end{multline*} 
where the infimum in $\ell$ is taken on measurable classification 
functions $\ell$, where $\ell^{\star}_q:\C{X}\mapsto \dlb 1, k \drb $ is the best classification function
for a fixed $q=(q_1,\dots,q_k)$ defined as
\[
\ell^{\star}_q(x) = \arg \min_{j \in \dlb 1, k \drb} \C{K}(q_j, p_x), \qquad x \in \C{X},
\]
whereas $q^{\star,\ell}_1, \dots, q^{\star,\ell}_k$ are
the best information $k$-means centers
with respect to $\ell(X)$
defined as 
\[
q^{\star,\ell}_j =
Z_j^{-1} \; \exp \Bigl\{ \B{P}_{X \, | \, \ell(X) = j}
 \bigl[ \log ( p_X ) \bigr] \Bigr\}, \qquad j \in \dlb 1, k \drb,
\]
where
\[
 Z_{j} = \int \exp \Bigl\{ \B{P}_{X \, | \, \ell(X) = j} 
\bigl[ \log ( p_X ) \bigr] \Bigr\} \, \ud\nu,
\]
with the convention that $q^{\star, \ell}_{j}$ can 
be given any arbitrary value in the case when $Z_j = 0$,
the corresponding criterion being in this case infinite.
Besides, we have the following Pythagorean identity
\[ 
\B{P}_X \Bigl( \C{K}(q_{\ell(X)}, p_X ) \Bigr)  
 = \B{P}_X \Bigl( \C{K}(q^{\star,\ell}_{\ell(X)}, p_X ) \Bigr)+
\B{P}_X \Bigl( \C{K} \bigl( q_{\ell(X)}, q^{\star,\ell}_{\ell(X)} 
\bigr) \Bigr). 
\] 
\end{prop}
\begin{proof}
This proposition is a straightforward consequence of Lemma \vref{lem:1.4}
applied to $\B{P}_{X, \, Y \, | \, \ell(X) = j}$.
\end{proof}

It may be of some help to state the empirical counterpart of 
the previous proposition, where formulas are somehow more 
explicit.
\begin{corollary}
 Let $X_1,\dots,X_n$ be an i.i.d sample drawn from $\B{P}_X$.
 Then, the empirical version of the information $k$-means problem
 tries to partition the observations $p_{X_1},\dots,p_{X_n}$ into
 $k$-clusters, what is expressed here by
 \begin{multline*}\inf_{ q \in \bigl( \B{L}^1_{+,1}(\nu) \bigr)^k} \frac{1}{n} \sum_{i=1}^n 
\min_{j \in \dlb 1, k \drb} \C{K} \bigl( q_{j}, p_{X_i} \bigr)=
\inf_{\ell : \dlb 1, n \drb \rightarrow \dlb 1, k \drb} \,
 \inf_{ q \in \bigl( \B{L}^1_{+,1}(\nu) \bigr)^k} \;\frac{1}{n} \sum_{i=1}^{n} 
 \C{K}\bigl( q_{\ell(i)},p_{X_i} \bigr)\\=
 \inf_{ q \in \bigl( \B{L}^1_{+,1}(\nu) \bigr)^k} 
 \inf_{\ell : \dlb 1, n \drb \rightarrow \dlb 1, k \drb} \,
 \;\frac{1}{n} \sum_{i=1}^{n} \C{K}\bigl( q_{\ell(i)},p_{X_i} \bigr)=  \inf_{ q \in \bigl( \B{L}^1_{+,1}(\nu) \bigr)^k} 
 \;\frac{1}{n} \sum_{i=1}^{n} \C{K} \bigl( q_{\ell_q^\star(i)}, p_{X_i} \bigr) \\
 = \inf_{\ell : \dlb 1, n \drb \rightarrow \dlb 1, k \drb} \, \frac{1}{n} \sum_{i=1}^n 
\C{K} \bigl( q^{\star,\ell}_{\ell(i)}, p_{X_i} \bigr) 
= \inf_{ \ell : \dlb 1, n \drb \rightarrow \dlb 1, k \drb}\;  \sum_{j=1}^{k}  \frac{\bigl\lvert \ell^{-1}(j) \bigr\rvert}{n}
\; \log \bigl( Z^{-1}_{j} \bigr),
\end{multline*} 
% \end{block}
where $\ell^{\star}_q:\C{X}\mapsto \dlb 1, k \drb $ is the best classification function
for a fixed $q=(q_1,\dots,q_k)$ defined as
\[ 
\ell^{\star}_q(i) = \arg \min_{j \in \dlb 1, k \drb} \C{K} (q_j, p_{X_i})
\] 
whereas $q^{\star,\ell}_{j}, j \in \dlb 1, k \drb$ are
the information $k$-means centers defined as 
\[
q^{\star,\ell}_{j} =
Z^{-1}_{j} \,  \Biggl(
\prod_{i \in \ell^{-1}(j)} p_{X_i}
\Biggr)^{1/|\ell^{-1}(j)|},
\]
where 
\[
 Z_{j}= \int \Biggl(\prod_{i \in \ell^{-1}(j)} p_{X_i}
\Biggr)^{1/|\ell^{-1}(j)|} \, \ud \nu.
\]
\end{corollary}
\begin{proof}
Apply the previous proposition to the empirical measure
 $\ds \ov{\B{P}}_X = \frac{1}{n} 
\sum_{i=1}^n \delta_{X_i}$ of the sample 
$X_1,\dots,X_n$.
\end{proof}  

We will now see that when the sample is in $\B{L}^2(\nu)$
and has a finite second moment, the optimal centers for 
a given classification function are also in $\B{L}^2(\nu)$, 
so that the optimization of the centers can be reduced to 
this space.
\begin{lemma}
\label{lem:1.6}
Let us assume that $\B{P}_X \Bigl( \int p_X^2 \, \ud \nu \Bigr) < \infty$. 
Then, the optimal centers $q^{\star,\ell}_{j}$ defined in the previous lemma
verify $q^{\star,\ell}_{j}\in \B{L}^2(\nu)$. Furthermore,
in this case
\begin{multline*}	
\inf \; \biggl\{ \; \B{P}_X 
\Bigl( \min_{j \in \dlb 1, k \drb} \C{K}( q_j, p_X ) \Bigr)
\, : \, q \in \Bigl( \B{L}_{+,1}^1(\nu) \Bigr)^k  \; \biggr\}
\\ = \inf \; \biggl\{ \; \B{P}_X 
\Bigl( \min_{j \in \dlb 1, k \drb} \C{K}( q_j, p_X ) \Bigr)
\, : \, q \in \Bigl( \B{L}_{+,1}^1(\nu) \cap \B{L}^2(\nu) \Bigr)^k \; \biggr\}. 
\end{multline*}
\end{lemma}
\begin{proof}
Apply Jensen's inequality and the Fubini-Tonelli theorem to
obtain that $q^{\star,\ell}_{j} \in \B{L}^2(\nu)$.
Indeed, for any $j \in \dlb 1, k \drb$, if $Z_j = 0$, we can pick up 
any value for $q^{\star,\ell}_j$, and in paticular a value 
in $\B{L}^2(\nu)$, in the same way if $
\B{P}_{X}( \ell(X) = j) = 0$, we can make 
an arbitrary choice for $q^{\star, \ell}_j$, otherwise, $Z_j > 0$,  
and
\begin{multline*}
\int (q^{\star,\ell}_{j})^2 \, \ud \nu = Z_{j}^{-2} \int \exp \biggl\{ 2 
\B{P}_{X \, | \, \ell(X) = j}  \bigl[ \log ( p_X ) 
\bigr] \biggr\} \, \ud \nu \\ \leq Z_{j}^{-2} \; \B{P}_{X \, | \, \ell(X) 
= j}  \biggl( 
\int p_X^2 \, \ud \nu 	\biggr) 
\leq Z_j^{-2} \B{P}_{X} \bigl( \ell(X) = j \bigr)^{-1} 
\B{P}_{X} \biggl( \int p_X^2 \, \ud \nu \biggr) < \infty 
\end{multline*}
Then according to Proposition \vref{prop:1.5} 
\begin{multline*}
\B{P}_X \Bigl[ \min_{j \in \dlb 1, k \drb} \C{K} \bigl( q_j, p_X \bigr) \Bigr] 
= \inf_{\ell:\C{X}\mapsto \dlb 1, k \drb} \B{P}_X \Bigl[ \C{K} \bigl( q_{\ell(X)}, p_X \bigr) \Bigr] 
\\ \geq \inf_{\ell:\C{X}\mapsto \dlb 1, k \drb}\B{P}_X 
\Bigl[ \C{K} \bigl( q^{\star,
\ell}_{\ell(X)}, p_X \bigr) \Bigr] \geq 
\inf_{\ell:\C{X}\mapsto \dlb 1, k \drb}
\B{P}_X \Bigl[ \min_{j \in \dlb 1, k \drb} \C{K} \bigl( q^{\star,\ell}_j, 
p_X \bigr) \Bigr],
\end{multline*}
showing that we can restrict the optimization to $q_j \in \B{L}^2(\nu)$.
\end{proof}

\section{PAC-Bayesian generalization bounds for the
linear $k$-means criterion} 

In this section, we derive non asymptotic generalization 
bounds for the linear $k$-means criterion defined hereafter. %of Proposition \vref{prop:7.10}. 

\begin{dfn}
\label{def:linear}
Given a random vector $W$ in a separable Hilbert space $H$ and 
a bounded measurable set of parameters $\Theta \subset H^k$,
the $k$-means linear criterion is defined as
\[ 
\B{P}_W \Bigl( \min_{j \in \dlb 1, k \drb} \langle \theta_j, W \rangle 
\Bigr), \qquad \theta \in \Theta.
\]  
If $W_1, \dots, W_n$ are $n$ independent copies of $W$,
the empirical linear $k$-means criterion is defined by taking 
the expectation with respect to the empirical measure
$\oB{P}_W = \frac{1}{n} \sum_{i = 1}^n \delta_{W_i}$ instead of 
integrating with respect to $\B{P}_W$.
\end{dfn}
Using a change of representation based on the kernel trick, 
we will show that all the criteria we defined so far
can be rewritten as linear $k$-means criteria in suitable
spaces of coordinates.

Consequently, our approach will be to prove a generalization 
bound for the linear $k$-means criterion and to study its 
consequences for the other criteria. 

To reach a $\C{O} \bigl( \sqrt{k / n} \bigr)$ speed up to 
logarighmic factors, we will
borrow ideas from the classical chaining method used to upper bound the expected supremum of Gaussian processes (see \cite{boucheron}).
However, we will 
transpose the idea of chaining into the setting of 
PAC-Bayesian deviation inequalities.
To obtain dimension free bounds, we will use a sequence 
of perturbations of the parameter by isonormal processes
with a variance parameter ranging in a geometric grid.
This multiscale perturbation scheme will play the same 
role as the $\delta$-nets in classical chaining.

Let us begin with an existence result.
\begin{prop}
\label{prop:exists}
In the setting of Definition \vref{def:linear}, let us assume 
that $\lVert W \rVert_{\infty} = \ess \sup_{\B{P}_W} \lVert W \rVert < 
+ \infty$. There is $\theta^* \in \ov{\Theta}$, the weak closure of 
$\Theta$, such that
\[
\B{P}_W \Bigl( \min_{j \in \dlb 1, k \drb} \langle \theta^*_j, W \rangle 
\Bigr) = \inf_{\theta \in \Theta} \Bigl[  
\B{P}_W \Bigl( \min_{j \in \dlb 1, k \drb} \langle \theta_j, W \rangle 
\Bigr) \Bigr]. 
\]
Moreover $\lVert \theta^* \rVert \leq \lVert \Theta \rVert 
\overset{\text{\rm def}}{=} \sup_{\theta \in \Theta} \lVert \theta \rVert$.
\end{prop}
\begin{proof}
%coucou
This is inspired by the proof of Theorem 3.2 in \cite{Fischer}.
Let us begin with the second statement. Since 
\[
\theta \longmapsto \lVert 
\theta \rVert = \sup_{\theta' \in H^k, \, \lVert \theta' \rVert = 1} 
\langle \theta', \theta\rangle
\]
is weakly lower semicontiuous, 
$\lVert \ov{\Theta} \rVert \leq \lVert \Theta \rVert$, so that in 
particular $\lVert \theta^* \rVert \leq \lVert \Theta \rVert$.
Moreover, for any $w \in H$, 
\begin{align*}
H^k & \longrightarrow \B{R}\\
\theta & \longmapsto \min_{j \in \dlb 1, k \drb}
\langle \theta_j, w \rangle 
\end{align*}
is weakly continuous, since, by definition of the weak topology of $H^k$, 
$\theta \mapsto \langle \theta_j, w \rangle$ are weakly continuous, and taking a finite 
minimum is a continuous operation.

Let $(\theta_n)_{n \in \B{N}}$ be a bounded 
sequence in $H^k$, converging weakly to $\theta$. By the dominated convergence theorem
\begin{align*}
 \lim_{n \rightarrow \infty}
\B{P}_W \Bigl( \, \min_{j \in \dlb 1, k \drb} \langle \theta_{n, \, j}, W \rangle \Bigr)
&=  \B{P}_W \Bigl(\lim_{n \rightarrow \infty} \, \min_{j \in \dlb 1, k \drb} 
\langle \theta_{n,j}, W \rangle \Bigr)\\
&= \B{P}_W\Bigl( \, \min_{j \in \dlb 1, k \drb} \langle \theta_j, 
W \rangle \Bigr),
\end{align*}
since $\ds \bigl\lvert \min_{j \in \dlb 1, k \drb} \langle \theta_{n,j}, W \rangle 
\bigr\rvert
\leq \rVert \theta_n \rVert \, \lVert W \rVert_{\infty}$. 
Thus
\[
\C{R} : \theta \longmapsto \B{P}_W \bigl( \min_{j \in \dlb 1, k \drb} \langle 
\theta_j, W \rangle \bigr)
\]
is weakly continuous on $\ov{\Theta}$. But the unit ball, and therefore
any ball of $H^k$, is weakly compact, so that $\ov{\Theta}$ being weakly
closed and bounded is also weakly compact. Consequently, $\C{R}$ 
reaches its minimum on $\ov{\Theta}$ at some (non necessarily unique)
point $\theta^* \in \ov{\Theta}$. Therefore 
\[
\C{R}(\theta^*) = \inf_{\theta \in \ov{\Theta}} \C{R}(\theta)
= \inf_{\theta \in \Theta} \C{R}( \theta),
\]
the last equality being due to the fact that $\C{R}$ is weakly 
continuous.

Note that we used the weak topology, since the
unit ball of $H^k$ is not strongly compact when the dimension 
of $H$ is infinite.
\end{proof}

We will prove generalization bounds based on the following PAC-Bayesian lemma. 
We will use it as a workhorse to produce all the deviation
inequalities necessary to achieve 
our goals. Combined with Jensen's inequality, it will also 
produce bounds in expectation.  

\begin{lemma}
\label{lem:1.9}
Consider two measurable spaces $\C{T}$ and $\C{W}$,
a prior probability measure 
$\pi \in \C{M}_+^1 (\C{T})$ defined on $\C{T}$, 
and a measurable function $h : \C{T} \times \C{W}  
\rightarrow \B{R}$. Let $W \in \C{W}$ be a random 
variable and let $(W_1, \dots, W_n)$ be a sample 
made of $n$ independent copies of $W$. 
Let $\lambda$ be a positive real parameter.
\begin{multline}
\label{eq:PAC1}
\B{P}_{W_1, \, \dots, \, W_n} \Biggl\{ \exp \Biggl[ 
\sup_{\rho \in \C{M}_+^1 ( \C{T} )} 
\sup_{\eta \in \B{N}} \biggl\{ \; \int \min \Bigl\{ \eta,  \; - 
\lambda \sum_{i=1}^n h(\theta', W_i) 
\\ - n \log \Bigl[ \B{P}_W \exp \bigl[ - \lambda h(\theta', W) \bigr] \Bigr] 
\Bigr\} \, \ud \rho (\theta') - \C{K}(\rho, \pi) \biggr\} \Biggr] \Biggr\} \leq 1.  
\end{multline}
Consequently, for any $\delta \in ]0,1[$, 
with probability at least $1 - \delta$, 
\begin{multline}
\label{eq:PAC2}
\sup_{\rho \in \C{M}_+^1 ( \C{T})} 
\sup_{\eta \in \B{N}} \biggl\{ \; \int \min \Bigl\{ \eta,  \; - 
\lambda \sum_{i=1}^n h(\theta', W_i) 
\\ - n \log \Bigl[ \B{P}_W \exp \bigl[ - \lambda h(\theta', W) \bigr] \Bigr] 
\Bigr\} \, \ud \rho (\theta') - \C{K}(\rho, \pi) \biggr\} \leq \log(\delta^{-1}).  
\end{multline}
\end{lemma} 
Note that the role of $\eta$ in this formula is to give a meaning to 
the integration with respect to $\rho$ in all circumstances.

\begin{proof}
%mianmian
We follow here the same arguments as in the proof of Proposition 1.7 in \cite{giulini2015}.
Remark that 
the supremum in $\rho$ can be restricted to the case when 
$\C{K}(\rho, \pi) < \infty$, 
and recall that in this case $\rho \ll \pi$ and 
$\ds \C{K}(\rho, \pi) = \int \log \biggl( \frac{\ud \rho}{\ud \pi} (\theta') 
\biggr) \, \ud \rho ( \theta')$. Note also that
\[
\int \B{1} \Bigl( \frac{\ud \rho}{\ud \pi} (\theta') > 0 \Bigr) \, 
\ud \rho (\theta') = 
\int \B{1} \Bigl( \frac{\ud \rho}{\ud \pi} (\theta') > 0 \Bigr) 
\frac{\ud \rho}{\ud \pi}(\theta') \, 
\ud \pi (\theta') 
= \int 
\frac{\ud \rho}{\ud \pi}(\theta') \, 
\ud \pi (\theta') = 1.
\]
Applying Jensen's inequality, we get 
\begin{multline*}
\exp \Biggl\{ 
\sup_{\rho \in \C{M}_+^1(\C{T}) } 
\sup_{\eta \in \B{N}} \int \min \biggl\{ \eta,  \; - 
\lambda \sum_{i=1}^n h(\theta', W_i) 
\\ \shoveright{ - n \log \Bigl[ \B{P}_W \exp \bigl[ - \lambda h(\theta', W) \bigr] \Bigr] 
\biggr\} \, \ud \rho (\theta') - \C{K}(\rho, \pi) 
\Biggr\} }\\ \shoveleft{ \leq   \sup_{\eta \in \B{N}}  \sup_{
\substack{\rho \in \C{M}_+^1(\C{T})
\\ \C{K}(\rho, \pi) < \infty}} 
\int \exp \Biggl\{ 
\min \biggl\{ \eta,  \; - 
\lambda \sum_{i=1}^n h(\theta', W_i) }
\\ \shoveright{ - n \log \Bigl[ \B{P}_W \exp \bigl[ - \lambda h(\theta', W) \bigr] \Bigr] 
\biggr\}  
\Biggr\} \frac{\ud \rho}{\ud \pi} (\theta')^{-1} \, \ud \rho (\theta') }
\\ 
\shoveleft{= \sup_{\eta \in \B{N}}  \sup_{\substack{
\rho \in \C{M}_+^1(\C{T})\\ \C{K}(\rho, \pi) < \infty}} 
\int \exp \Biggl\{ 
\min \biggl\{ \eta,} \\ \shoveright{ - 
\lambda \sum_{i=1}^n h(\theta', W_i) 
- n \log \Bigl[ \B{P}_W \exp \bigl[ - \lambda h(\theta', W) \bigr] \Bigr] 
\biggr\}  
\Biggr\} \B{1} \biggl( \frac{\ud \rho}{\ud \pi} (\theta') > 0 \biggr) \, \ud 
\pi (\theta')} 
\\ 
\shoveleft{\leq 
\sup_{\eta \in \B{N}}   
\int \exp \Biggl\{ 
\min \biggl\{ \eta,  \; - 
\lambda \sum_{i=1}^n h(\theta', W_i) 
- n \log \Bigl[ \B{P}_W \exp \bigl[ - \lambda h(\theta', W) \bigr] \Bigr] 
\biggr\}  \, \ud \pi (\theta') }
\\ 
\underset{\substack{\text{monotone}\\\text{convergence}}}{=}
\int \exp \Biggl\{ 
- \lambda \sum_{i=1}^n h(\theta', W_i) 
- n \log \Bigl[ \B{P}_W \exp \bigl[ - \lambda h(\theta', W) \bigr] \Bigr] 
\Biggr\} \, \ud \pi(\theta').  
\end{multline*} %Olivier
Let us put 
\begin{multline*}
Y' = 
\sup_{\rho \in \C{M}_+^1 (\C{T})} 
\sup_{\eta \in \B{N}} \biggl\{ \; \int \min \Bigl\{ \eta,  \; - 
\lambda \sum_{i=1}^n h(\theta', W_i) 
\\ \shoveright{ - n \log \Bigl[ \B{P}_W \exp \bigl[ - \lambda 
h(\theta', W) \bigr] \Bigr] 
\Bigr\} \, \ud \rho (\theta') - \C{K}(\rho, \pi) \biggr\} \text{ and }} \\ 
\shoveleft{Y = 
\log \int \exp \Biggl\{ 
- \lambda \sum_{i=1}^n h(\theta', W_i) 
- n \log \Bigl[ \B{P}_W \exp \bigl[ - \lambda h(\theta', W) \bigr] \Bigr] 
\Biggr\} \, \ud \pi(\theta').\hfill}
\end{multline*}
We just proved that $Y' \leq Y$. Moreover, $Y$ is measurable, 
according to Fubini's theorem for non-negative functions. 
Therefore $Y$ is a random variable. Note that we did not prove 
that $Y'$ itself is measurable. 
Remark now that
\begin{multline*}
\B{P}_{W_1, \, \dots \, ,W_n} \bigl[ \exp (Y) \bigr] \\ =  
\B{P}_{W_1, \, \dots \, ,W_n}  \int \exp \Biggl\{ 
- \lambda \sum_{i=1}^n h(\theta', W_i) 
- n \log \Bigl[ \B{P}_W \exp \bigl[ - \lambda h(\theta', W) \bigr] \Bigr] 
\Biggr\} \, \ud \pi(\theta'),  
\\ \underset{\text{Fubini}}{=} \int \B{P}_{W_1, \, \dots \, ,W_n} \exp \Biggl\{ 
- \lambda \sum_{i=1}^n h(\theta', W_i) 
- n \log \Bigl[ \B{P}_W \exp \bigl[ - \lambda h(\theta', W) \bigr] \Bigr] 
\Biggr\} \, \ud \pi(\theta') \\
= \scalebox{1.2}[1.5]{$\ds\int$} \left( \B{1} \biggl( 
\B{P}_W \Bigl[ \exp \Bigl( - \lambda h(\theta', W) \Bigr) 
\Bigr] < + \infty \biggr)
\prod_{i=1}^n \frac{\B{P}_{W_i} \bigl[ \exp \bigl( - \lambda 
h(\theta', W_i ) \bigr) \bigr]}{
\B{P}_{W} \Bigl[ \exp \Bigl( - \lambda h(\theta', W ) \bigr) \bigr]} 
\right) 
\ud \pi(\theta')
\leq 1,  
\end{multline*} 
proving the first part of the lemma.
From Markov's inequality, 
\[
\B{P} \bigl( Y \geq \log(\delta^{-1}) \bigr) \leq 
\delta \, \B{P}_{W_1, \, \dots \, ,W_n} \bigl[ \exp(Y) \bigr] \leq \delta. 
\]
Consequently $\B{P} \bigl( Y \leq \log(\delta^{-1}) \bigr) \geq 1 - \delta$. 
We have proved that the non necessarily measurable event $Y' \leq 
\log(\delta^{-1})$ 
contains the measurable event $Y \leq \log(\delta^{-1})$ whose 
probability is at least $1 - \delta$. 
\end{proof}

We are now ready to state and prove our generalization bounds for the linear 
$k$-means criterion.

\begin{lemma}
\label{lem:14.1}
Let $W$ be a random vector in a separable Hilbert space $H$.
Let $(W_1, \dots, W_n)$ be a sample made of $n$ independent 
copies of $W$. Let $\Theta \subset H^k$ be a bounded 
measurable set of parameters. Define 
\[
\lVert \Theta \rVert = \sup \biggl\{ 
\biggl( \sum_{j=1}^k \lVert \theta_j \rVert^2 
\biggr)^{1/2} \, : \, \theta \in \Theta \biggr\} < \infty
\]
and assume that, for some real valued parameters $a$ and $b$, 
\[
\B{P}_W \Bigl( \min_{j \in \dlb 1, k \drb} \langle \theta_j, 
W \rangle \in [a,b] \text{ for any } \theta \in \Theta \Bigr) = 1.
\] 
Assume also that $\ds \lVert W \rVert_{\infty} \overset{\text{\rm def}}{=} 
\ess \sup_{\B{P}_W} 
\, \lVert W \rVert < \infty$.

Our first result gives an observable upper bound for the
$k$-means criterion, provided that the above parameters are 
known or upper bounded by known quantities.

For any $k \geq 2$, any $n \geq 2 k$ and any $\delta \in ]0,1[$, with 
probability at least $1 - \delta$, for any $\theta 
\in \Theta$,
\begin{multline*}
\B{P}_W \Bigl( \min_{j \in \dlb 1, k \drb} 
\langle \theta_j, W \rangle \Bigr) \leq 
\oB{P}_W \Bigl( 
\min_{j \in \dlb 1, k \drb} \langle \theta_j, W \rangle \Bigr) \\ + 
\Biggl( \frac{\log(n/k)}{\log(2)} \sqrt{\frac{8 \log(k)}{n}} + 
2 \sqrt{\frac{\log(k)}{n}} \; \Biggr)
\lVert \Theta \rVert \lVert W \rVert_{\infty}  
\\ + 
\sqrt{\frac{(\sqrt{2} + 1 )\Bigl( k (b-a)^2 + 2 \log(ek) 
\lVert W \rVert_{\infty}^2 \lVert \Theta \rVert^2 
\Bigr) }{n}} + \sqrt{\frac{\log(\delta^{-1})}{2n}}(b-a),
\end{multline*}
where $\ds \oB{P}_W = \frac{1}{n} \sum_{i=1}^n \delta_{W_i}$
is the empirical measure.

Our second result deals with the excess risk with respect
to a non random reference parameter $\theta^* \in \Theta$.

If $\theta^* \in \Theta$ is a non random value of the parameter,
with probability at least $1 - \delta$, 
for any $\theta \in \Theta$, 
\begin{multline*}
\Bigl( \B{P}_W - \oB{P}_W \Bigr) \Bigl( 
\min_{j \in \dlb 1, k \drb} \langle \theta_j, W \rangle -
\min_{j \in \dlb 1, k \drb} \langle \theta_j^*, W \rangle \Bigr) 
\\ \leq
\Biggl( \frac{\log(n/k)}{\log(2)} \sqrt{\frac{8 \log(k)}{n}} + 
2 \sqrt{\frac{\log(k)}{n}} \; \Biggr)
\lVert \Theta \rVert \lVert W \rVert_{\infty}  
\\ 
+  \sqrt{\frac{(\sqrt{2} + 1 )\Bigl( k (b-a)^2 + 2 \log(ek) 
\lVert W \rVert_{\infty}^2 \lVert \Theta \rVert^2 
\Bigr)}{n}}
+ \sqrt{\frac{2 \log(\delta^{-1})}{n}} (b-a).
\end{multline*}
Our third result draws the consequences of this excess 
risk bound for an $\epsilon$-minimizer
$\wh{\theta}$.

In the case when the estimator $\wh{\theta}(W_1, \dots, W_n)
\in \Theta$
is such that $\B{P}_{W_1, \dots, \, W_n}$ almost surely
\[
\oB{P} \Bigl( \min_{j \in \dlb 1, k \drb} \langle \wh{\theta}_j, 
W \rangle \Bigr)
\leq \inf_{\theta \in \Theta} \oB{P} \Bigl( \min_{j \in \dlb 1, k \drb} \langle \theta_j, 
W \rangle \Bigr) + \epsilon,
\]
$\ds \B{P}_W 
\Bigl( \min_{j \in \dlb 1, k \drb} \langle \wh{\theta}_j, W
\rangle \Bigr) - \inf_{\theta \in \Theta} \B{P}_W 
\Bigl( \min_{j \in \dlb 1, k \drb} \langle \theta_j, W
\rangle \Bigr) - \epsilon$ satisfies the same bound with at least the same 
probability.\\[1ex] 
Moreover, the expected excess risk satisfies 
\begin{multline*}
\B{P}_{W_1, \dots, W_n} \Bigl[ \B{P}_W \Bigl( \min_{j \in \dlb 1, k \drb} 
\langle \wh{\theta}_j, W \rangle \Bigr) - \inf_{\theta \in \Theta} \B{P}_W 
\Bigl( \min_{j \in \dlb 1, k\drb} \langle \theta_j, W \rangle 
\Bigr) \Bigr] \\ \leq \Biggl( \frac{\log(n/k)}{\log (2)} \sqrt{ \frac{
8 \log(k)}{n}} + 2 \sqrt{\frac{\log(k)}{n}} \; \Biggr) \lVert 
\Theta \rVert \, \lVert W \rVert_{\infty} 
\\ + \sqrt{ \frac{ \bigl( \sqrt{2} + 1 \bigr) \Bigl( 
k (b-a)^2 + 2 \log(ek) \lVert W \rVert_{\infty}^2 \lVert 
\Theta \rVert^2 \Bigr)}{n}} + \epsilon.
\end{multline*}
\end{lemma}
\begin{proof}
Assume without loss of generality that $H = \ell^2 \subset \B{R}^{\B{N}}$.
Let 
\[
\rho_{\theta' \, | \, \theta} = \B{P}_{\theta_i + 
\beta^{-1/2} \epsilon_i, i \in \B{N}}, \quad \theta \in \B{R}^{\B{N}},
\]
be a Gaussian conditional probability distribution with values in 
$\C{M}_+^1 \bigl( \B{R}^{\B{N}} \bigr)$, where $\epsilon_i$, $i \in \B{N}$ is an infinite 
sequence of independent standard normal random variables. 
When $\theta$ and $\theta' \in \B{R}^{k \times \B{N}}$ are
made of $k$ infinite sequences of real numbers, let 
\[
\rho_{\theta' \, | \, \theta} = \bigotimes_{j=1}^k \rho_{\theta_j' \, | \, 
\theta_j}
\] 
be the tensor product of the previously defined conditional probability 
distributions.
Let $W$ be a random vector in the separable Hilbert space $\ell^2 
\subset \B{R}^{\B{N}}$.
Consider the measurable functions
\[
f(\theta, w) = \min_{j \in \dlb 1, k \drb} \langle \theta_j, w \rangle, 
\qquad \theta \in \B{R}^{k \times \B{N}}, \, w \in \ell^2, 
\] 
where the scalar product is extended beyond $\ell^2$ 
as follows. For any $u, v \in \B{R}^{\B{N}}$, let us define 
$\langle u, v \rangle$ as 
\[
\langle u, v \rangle = 
\begin{cases}
\ds \lim_{s \rightarrow \infty} \sum_{t=0}^s u_t v_t, & 
\ds \text{when } \limsup_{s \rightarrow \infty} \sum_{t=0}^s u_t v_t = 
\liminf_{s \rightarrow \infty} \sum_{t=0}^s u_t v_t \in \B{R}.\\
0, & \text{otherwise.} 
\end{cases}
\] 
Remark that this extension is measurable, but not bilinear.

Our strategy will be to decompose the opposite of the 
centered empirical risk \linebreak $\bigl( \B{P}_W - \oB{P}_W \bigr) 
\bigl[ f(\theta, W) \bigr]$ into 
\begin{multline}
\label{eq:strat}
\bigl( \B{P}_W - \oB{P}_W \bigr) \bigl[ f(\theta, W) \bigr] = 
\bigl( \B{P}_W - \oB{P}_W \bigr) \bigl( \delta_{\theta' | \theta} 
- \underbrace{\rho_{\theta' \, | \, \theta}}_{\makebox[0pt]{
\scriptsize small 
perturbation}} \bigr) \bigl[ f(\theta', W) \bigr]\\
+ \sum_{q=1}^p \bigl( \B{P}_W - \oB{P}_W  \bigr) \bigl( \rho^{2^{q-1}}_{
\theta' | \theta} - \underbrace{
\rho^{2^q}_{\theta' \, | \, \theta}
}_{
\makebox[0pt]{\scriptsize chain of intermediate scales}} \bigr) \bigl[ 
f(\theta', W) \bigr] \\
+ \bigl( \B{P}_W - \oB{P}_W \bigr) \underbrace{\rho^{2^p}_{\theta' \, | \, 
\theta}}_{\makebox[0pt]{\scriptsize big perturbation}}  
\bigl[ f( \theta', W) \bigr],  
\end{multline} 
where $\delta_{\theta' \, | \, \theta}$ is the Dirac (or identity)
transition kernel and 
$\rho^{2^q}_{\theta' \, | \, \theta}$ is the transition kernel
$\rho_{\theta' \, | \, \theta}$ iterated $2^q$ times.

Let 
\[
\ov{f}(\theta, w) = f(\theta, w) - \B{P}_W \bigl( f(\theta, W) \bigr),
\qquad \theta \in \B{R}^{k \times \B{N}}, w \in \ell^2,
\]
be the centered loss function.

We will first apply the PAC-Bayesian inequalities of Lemma 
\vref{lem:1.9} to the function
\[
h(\theta', w) = \bigl( \delta_{\theta'' \, | \, \theta'} - 
\rho_{\theta'' \, | \, \theta'} \bigr) \bigl[ \, 
\ov{f} ( \theta'', w ) \bigr] 
= \ov{f} ( \theta', w ) - \rho_{\theta'' \, | \, \theta'} \bigl[ 
\, 
\ov{f} ( \theta'', w ) \bigr], \quad \theta' \in \B{R}^{k 
\times \B{N}}, \; w \in \ell^{2}
\]  
and to the reference measure $\pi_{\theta'} 
= \rho_{\theta' \, | \, \theta = 0}$.
\begin{lemma}
\label{lem:15}
The function $h$ satisfies 
\[
\bigl( \pi_{\theta'} \otimes \B{P}_W \bigr) \Bigl( 
\lvert h(\theta', W) \rvert \leq 2 \sqrt{2 \log(k) / \beta} \, 
\lVert W \rVert_{\infty} \Bigr) = 1,
\]
where $\ds \lVert W \rVert_{\infty} = \ess \sup_{\B{P}_W} \lVert W \rVert$.
\end{lemma}
\begin{proof}
Remark that for any $w \in \ell^2$, $\pi_{\theta'}$ almost surely, 
\begin{multline*}
\bigl( \delta_{\theta'' \, | \, \theta'} - \rho_{\theta'' \, | \, \theta'}
\bigr) \bigl[ f(\theta'', w) \bigr] = \rho_{\theta'' \, | \theta'} \Bigl( 
\min_{j} \langle \theta'_j, w \rangle - \min_j \langle \theta''_j, w 
\rangle \Bigr) \\ \leq \rho_{\theta'' \, | \, \theta'} \Bigl( \max_j 
\langle \theta''_j - \theta'_j , w \rangle \Bigr), 
\end{multline*}
since in this situation, the first case in the extended
definition of the scalar
product applies with probability one (according to Kolmogorov's 
three series theorem).
Considering that under $\rho_{\theta'' \, | \, \theta'}$, 
$\langle \theta''_j - \theta'_j, w \rangle,
\, j \in \dlb 1, k \drb$ are $k$ independent centered 
real normal random variables with variance $\lVert w \rVert^2 / \beta$  
and applying a classical maximal inequality for the expectation
of the maximum of $k$ standard normal variables 
(see section 2.5 in \cite{boucheron}), we get that
\[
\bigl( \delta_{\theta'' \, | \, \theta'} - \rho_{\theta'' \, | \, 
\theta'} \bigr) \bigl[ f( \theta', w ) \bigr] \leq \sqrt{2 \log(k) / \beta} \, 
\lVert w \rVert.
\]
Reasoning in a similar way for the opposite, we get
\[
- \bigl( \delta_{\theta'' \, | \, \theta'} - \rho_{\theta'' \, | \, 
\theta'} \bigr) \bigl[ f( \theta', w ) \bigr] \leq 
\rho_{\theta'' \, | \, \theta'} \Bigl( \max_j
\langle \theta'_j - \theta''_j, w \rangle \Bigr) \leq 
\sqrt{2 \log(k) / \beta} \, 
\lVert w \rVert.
\]
The lemma follows from the definition of $h$.
\end{proof} 

Applying Lemma \vref{lem:1.9} to $h : \B{R}^{k \times \B{N}} \times \ell^2 
\rightarrow \B{R}$, $\pi = \rho_{\theta' \, | \, 
\theta = 0}$ and restricting the supremum in $\rho \in \C{M}_+^1 \bigl(
\B{R}^{k \times \B{N}} \bigr)$ to $\rho \in \bigl\{ 
\rho_{\theta' \, | \, \theta} \, : \, \theta \in ( \ell^2 )^k
\bigr\}$, we get
\begin{multline*}
\B{P}_{W_1, \, \dots \, , W_n} \Biggl\{ \exp \sup_{\theta \in (\ell^2)^k} 
\Biggl[ n \lambda  
\bigl( \B{P}_W - \oB{P}_W \bigr) \bigl( \rho_{\theta' \, | \, \theta} 
- \rho_{\theta' \, | \, \theta}^{2} \bigr) f( \theta', W) \\ 
- n \rho_{\theta' \, | \, \theta} \biggl[ 
\log \biggl( \B{P}_W \Bigl[ \exp \Bigl( - \lambda 
h \bigl( \theta', W \bigr)
\Bigr] \biggr) \biggr]
- 
\frac{\beta \lVert \theta \rVert^2 }{ 
2 } \Biggl] \Biggr\} \leq 1, 
\end{multline*}
where we have let $\eta$ go to $+ \infty$, using monotone convergence 
(since $h$ is bounded from the previous lemma) and 
where we have computed
\begin{multline*}
\C{K} \bigl( \rho_{\theta' \, | \, \theta}, \pi \bigr) = 
\sum_{j = 1}^k \C{K} \bigl( \rho_{\theta'j \, | \, \theta_j}, 
\rho_{\theta'_j \, | \, 0} \bigr) \\ = 
\sum_{j=1}^k \sum_{i \in \B{N}} \C{K} \bigl( 
\C{N} ( \theta_{j,i}, \beta^{-1} ) , \C{N} ( 0, \beta^{-1}) \bigr) = 
\frac{\beta}{2} \sum_{j=1}^k \sum_{i \in \B{N}} \theta_{j,i}^2
= \frac{\beta \lVert \theta \rVert^2}{2}.
\end{multline*}
Apply now Jensen's inequality and devide by $n \lambda$ to get
\begin{multline*}
\B{P}_{W_1, \, \dots \, , W_n} \Biggl\{ \sup_{\theta \in \ell_2^k} 
\Biggl[ 
\bigl( \B{P}_W - \oB{P}_W \bigr) \bigl( \rho_{\theta' \, | \, \theta} 
- \rho_{\theta' \, | \, \theta}^{2} \bigr) f( \theta', W) \\ 
- \lambda^{-1} \rho_{\theta' \, | \, \theta} \biggl[ 
\log \biggl( \B{P}_W \Bigl[ \exp \Bigl( - \lambda 
h \bigl( \theta', W \bigr)
\Bigr] \biggr) \biggr]
- 
\frac{\beta \lVert \theta \rVert^2 }{ 
2 n \lambda} \Biggl] \Biggr\} \leq 0. 
\end{multline*}
From Hoeffding's inequality, since $\B{P}_W \bigl( h(\theta', W) \bigr) = 0$, 
$\pi_{\theta'}$ almost surely,
\[
\B{P}_{W} \Bigl[ \exp \Bigl( - \lambda h( \theta', W ) \Bigr) \Bigr] 
\leq \exp \Bigl( \frac{ \lambda^2}{2} \ess \sup_{\B{P}_W} h( \theta', W)^2
\Bigr) \leq \exp \Bigl( \frac{4 \lambda^2}{\beta} \log(k) \lVert W \rVert_{
\infty}^2 \Bigr). 
\]
Considering a measurable bounded subset 
$\Theta \subset (\ell^2)^k$, we deduce that 
\[
\B{P}_{W_1, \, \dots \, , W_n} \biggl[ \, \sup_{\theta \in \Theta} \, 
\bigl( \B{P}_W - \oB{P}_W \bigr) \bigl( \rho_{\theta' \, | \, \theta} 
- \rho_{\theta' \, | \, \theta}^{2} \bigr) f( \theta', W) \biggr] 
\\ \leq \frac{4 \lambda}{\beta} \log(k) \lVert W \rVert_{\infty}^2 + 
\frac{\beta \lVert \Theta \rVert^2 }{2 n \lambda}.
\]
In order to minimize the right-hand side, choose
\[
\lambda = \frac{\beta \lVert \Theta \rVert }{
\sqrt{8 n \log(k)}\lVert W \rVert_{\infty}}
\] 
and define
\begin{equation}
\label{eq:F}
F = \lVert W \rVert_{\infty} \lVert \Theta \rVert
\sqrt{\frac{ 8 \log(k) }{ n }}.
\end{equation}
We get 
\[
\B{P}_{W_1, \, \dots \, , W_n} \biggl[ \, \sup_{\theta \in \Theta} \, 
\bigl( \B{P}_W - \oB{P}_W \bigr) \bigl( \rho_{\theta' \, | \, \theta} 
- \rho_{\theta' \, | \, \theta}^{2} \bigr) f( \theta', W) \biggr] 
\leq F.
\]
For any integer $q$, the iterated transition kernel $\rho^{2^q}_{\theta' 
\, | \, \theta}$ is equal to $\rho_{\theta' \, | \, \theta}$ 
with $\beta$ replaced by $2^{-q} \beta$. As $F$ is independent of $\beta$,
we therefore deduce that 
\[ 
\B{P}_{W_1, \, \dots \, ,W_n} \biggl\{ 
\sup_{\theta \in \Theta} \Bigl[ 
\Bigl( \B{P}_W - \oB{P}_W \Bigr) \Bigl( \rho_{\theta' | \theta}^{2^{q-1}} 
- 
\rho_{\theta' | \theta}^{2^q} \Bigr) f \bigl( \theta', W \bigr) 
\Bigr] \biggr\} \leq F. 
\] 
Summing up for $q = 1$ to $p$, where $p$ is to be chosen later, and exchanging $\sum_q$ and $\sup_{\theta}$, we deduce that
\begin{multline*}
\B{P}_{W_1, \, \dots \, , W_n} \biggl\{ 
\sup_{\theta \in \Theta} \Bigl[ 
\Bigl( \B{P}_W - \oB{P}_W \Bigr) \Bigl( \rho_{\theta' \, | \, \theta} 
- \rho_{\theta' \, | \, \theta}^{2^p} \Bigr) f \bigl( 
\theta', W \bigr) \Bigr] \biggr\} \\ \leq  
\B{P}_{W_1, \, \dots \, , W_n} \biggl\{ 
\sum_{q=1}^{p} \sup_{\theta \in \Theta} \Bigl[ 
\Bigl( \B{P}_W - \oB{P}_W \Bigr) \Bigl( \rho^{2^{q-1}}_{\theta' \, | \, \theta} 
- \rho^{2^{q}}_{\theta' \, | \, \theta} \Bigr) f \bigl( \theta', W \bigr) \Bigr] \biggr\} \leq p F . 
\end{multline*}
As we are interested in bounding from above
$\bigl( \B{P}_W - \oB{P}_W \bigr) 
f(\theta, W)$, according to the decomposition 
formula \myeq{eq:strat}, there remains to upper bound
\begin{align}
\label{eq:39}
& \bigl( \B{P}_W - \oB{P}_{W} \bigr) \bigl( \delta_{ \theta' \, | \, \theta}
- \rho_{\theta' \, | \, \theta} \bigr) \bigl[ f(\theta', W) \bigr] \\ 
\text{and } 
& \bigl( \B{P}_W - \oB{P}_{W} \bigr) 
 \rho_{\theta' \, | \, \theta}^{2^p} \bigl[ f(\theta', W)
\bigr], \\ 
\text{or with } & \text{a change of notation }  \nonumber\\
\label{eq:41}
& \bigl( \B{P}_W - \oB{P}_{W} \bigr) 
 \rho_{\theta' \, | \, \theta} \bigl[ f(\theta', W) \bigr].
\end{align}
An almost sure bound for \eqref{eq:39} is provided by Lemma \vref{lem:15},
since \eqref{eq:39} is equal to $\oB{P}_W \bigl[ h(\theta, W) \bigr]$.
To bound \eqref{eq:41}, introduce the influence function
\begin{equation}
\label{eq:psi}
\psi(x) = \begin{cases}
\log \bigl( 1 + x + x^2/2 \bigr), & x \geq 0,\\
- \log ( 1 - x + x^2/2 \bigr), & x \leq 0
\end{cases}
\end{equation}
and put
\[
\wt{f} ( \theta, W ) = f(\theta, W) - \frac{a + b}{2}.
\]
The function $\psi$ is chosen to be symmetric and to satisfy
\begin{equation}
\label{eq:psySmall}
\psi(x) \leq \log \bigl( 1 + x + x^2 / 2 \bigr), \qquad x \in \B{R},
\end{equation}
since we can check that
\[ 
\log \bigl( 1 + x + x^2/2 \bigr) + 
\log \bigl( 1 - x + x^2/2 \bigr) = 
\log \bigl[ \bigl( 1 + x^2/2 \bigr)^2 - x^2 \bigr] = 
\log \bigl( 1 + x^4/4 \bigr) \geq 0. 
\] 
Decompose \eqref{eq:41} into 
\begin{align}
\Bigl( \B{P}_W - \oB{P}_{W} \Bigr) 
 \rho_{\theta' \, | \, \theta} f(\theta', W) & = 
\rho_{\theta' \, | \, \theta} \Bigl( \B{P}_W - \oB{P}_W \Bigr) 
\wt{f} ( \theta', W )
\nonumber\\  
\label{eq:42}
& = \rho_{\theta' \, | \, \theta} \Bigl[ \B{P}_W \wt{f}(\theta', W) - 
\oB{P}_W \Bigl( \lambda^{-1} \psi \bigl[ \lambda 
\wt{f}(\theta', W) \bigr] \Bigr) \Bigr] \\ 
\label{eq:43}
& \qquad + \rho_{\theta' \, | \, \theta} \oB{P}_W \Bigl[ \lambda^{-1} 
\psi \bigl[ \lambda \wt{f}(\theta', W) \bigr] - \wt{f} (
\theta', W) \Bigr].
\end{align}
In order to bound \eqref{eq:43}, note that from lemma 7.2 in \cite{catoni2012}
\begin{align}
\label{eq:2.718}
\bigl\lvert x - \psi(x) \bigr\rvert \leq \frac{x^2}{4(1 + \sqrt{2})}, \qquad 
x \in \B{R}.
\end{align}
Therefore, from the inequalities $(a+b)^2 \leq 2a^2 +2b^2$ and
$\min_j a_j - \min_j b_j \leq \max_j (a_j - b_j)$, so that
$(\min_j a_j - \min_j b_j)^2 \leq \max_j (a_j - b_j)^2$,
for any $\theta \in \bigl(\ell^2\bigr)^k$, 
$\B{P}_W$ almost surely,
\begin{multline*}
% \label{eq:44}
\rho_{\theta' \, | \, \theta} \Bigl[ \lambda^{-1} 
\psi \bigl[ \lambda \wt{f}(\theta', W) \bigr] - \wt{f} (
\theta', W) \Bigr] 
\leq \frac{\lambda}{4(1 + \sqrt{2})}
\rho_{\theta' \, | \, \theta} \bigl[ 
\wt{f}(\theta', W)^2 \bigr]  
\\ \leq \frac{\lambda}{2(1 + \sqrt{2})} \Bigl[ \bigl( \min_j 
\langle \theta_j, W \rangle - (a+b)/2 \bigr)^2 + \rho_{\theta' \, | \, \theta} 
\Bigl( \max_j \langle \theta'_j - \theta_j, W \rangle^2 \Bigr) \Bigr]. 
\end{multline*}
At this point, it remains to bound the variance term $\rho_{\theta' \, | \, 
\theta} 
\Bigl( \max_j \langle \theta'_j - \theta_j, W \rangle^2 \Bigr)$.
Let us remark that
\[
\rho_{\theta' \, | \, \theta} \circ 
\bigl( \theta' \mapsto \langle \theta_j'  - \theta_j, W \rangle_{j=1}^k \bigr)^{-1} 
= \C{N} \Bigl( 0, \lVert W \rVert^2 / \beta 
\Bigr)^{\otimes k}.
\]
In other words, under $\rho_{\theta' \, | \, \theta}$, 
the sequence $\bigl( \langle \theta'_j - \theta_j, W \rangle, 1 \leq j \leq k
\bigr)$ is made of $k$ independent centered normal random 
variables with variance 
$\lVert W \rVert^2 / \beta$.
Therefore, we need the following maximal inequality.
\begin{lemma}
\label{lem: 2.7182}
Let $(\varepsilon_1,\dots,\varepsilon_k)$ be a sequence of Gaussian random variables such that $\varepsilon_j \sim \C{N}(0,\sigma^2)$.
We have
\[
 \B{E}\big(\max_{1\leq j \leq k} \, \varepsilon_j^2 \big) \leq 2 \sigma^2 \log(ek).
\]
\end{lemma}
\begin{proof}
\begin{multline*}
 \B{E}\Bigl(\max_{1 \leq j \leq k} \, \varepsilon_j^2 \Bigr) =
\int_{\B{R}_{+}} \B{P}\Bigl(\max_{1 \leq j \leq k } \,  \varepsilon_j^2 >
\,t \Bigr) \, \ud t \\
\leq \int_{\B{R}_{+}} \min \biggl\{
\sum_{j=1}^k \B{P}\bigl(\varepsilon_j^2 >
\,t \bigr), 1 \biggr\} \, \ud t 
\leq \int_{\B{R}_{+}} \min \Bigl\{ 2 k \, \B{P}\bigl(\varepsilon_1 > \sqrt{t} 
\bigr),
1 \Bigr\} \, \ud t \\ \leq 
\int_{\B{R}_{+}} \min \Bigl\{ k \exp \Bigl(
- \frac{t}{2\sigma^2}  \Bigr), 1 \Bigr\} \, \ud t 
\leq 2\sigma^2 \log(k) + \int_{2 \sigma^2 \log(k)}^{+ \infty}
k \exp \Bigl(- \frac{t}{2\sigma^2}   \Bigr) \, \ud t \\
\leq 2 \sigma^2 \log(k) + 2 \sigma^2=2\sigma^2 \log(ek).
\end{multline*}
\end{proof}

Accordingly, we obtain $\B{P}_W$ almost surely,
\begin{multline}
\label{eq:44}
\rho_{\theta' \, | \, \theta} \Bigl[ \lambda^{-1} 
\psi \bigl[ \lambda \wt{f}(\theta', W) \bigr] - \wt{f} (
\theta', W) \Bigr]
\\ \leq \frac{\lambda}{2(1 + \sqrt{2})} \Bigl[ \bigl( \min_j 
\langle \theta_j, W \rangle - (a+b)/2 \bigr)^2 + \rho_{\theta' \, | \, \theta} 
\Bigl( \max_j \langle \theta'_j - \theta_j, W \rangle^2 \Bigr) \Bigr] 
\\ \leq \frac{\lambda}{2 (1 + \sqrt{2})} \bigl[ (b-a)^2/4 + 2 \log(ek) \lVert W
\rVert_{\infty}^2/\beta \bigr]. 
\end{multline}
The right-hand side of this inequality provides an almost sure upper bound 
for \eqref{eq:43}.
To bound \eqref{eq:42}, or rather the expectation of an exponential 
moment of \eqref{eq:42}, we can write a PAC-Bayesian bound using
the influence function $\psi$. According to Lemma \vref{lem:1.9}, 
\begin{multline*}
\B{P}_{W_1, \, \dots \, , W_n} \Biggl\{ 
\exp \sup_{\theta \in \Theta} \Biggl[ 
- n \lambda \rho_{\theta' \, | \, \theta} \oB{P}_{W} \Bigl( \lambda^{-1} 
\psi \bigr[ \lambda \wt{f} ( \theta', W ) \bigr] 
\Bigr)
\\ -n \rho_{\theta' \, | \, \theta} \biggl[ \log 
\biggl( \B{P}_W \Bigl[ \exp \Bigl( \psi \bigl[ - \lambda \wt{f} \bigl(  
\theta', W \bigr) \bigr] \Bigr) \Bigr] \biggr) \biggr] 
- \frac{ \beta \lVert \theta \rVert^2}{2} \Biggr] \Biggr\} \leq 1.
\end{multline*}
Indeed, it is easy to check that the integrand of $\rho_{\theta' \, | \, 
\theta}$ is integrable, so that we can apply the monotone convergence
theorem to remove $\eta$ from the equation produced by Lemma \ref{lem:1.9}.
Using the bound \myeq{eq:psySmall} 
and removing the exponential according to Jensen's inequality,
we obtain
\begin{multline*}
\B{P}_{W_1, \, \dots \, , W_n} \Biggl\{ 
\sup_{\theta \in \Theta} \Biggl[ 
\rho_{\theta' \, | \, \theta} \Bigl[  \B{P}_W \Bigl( \wt{f} \bigl( \theta', W \bigr) 
\Bigr)
-\oB{P}_W \Bigl( \lambda^{-1} \psi \bigl[ \lambda \wt{f}(\theta', W) \bigr] 
\Bigr) \Bigr]  
\\ - \frac{\lambda}{2} \rho_{\theta' \, | \, \theta} 
\Bigl[ \B{P}_W \Bigl( 
\wt{f} ( \theta', W )^2 \Bigr) \Bigr] \Biggr] \Biggr\}  \leq 
 \frac{\beta \lVert \Theta \rVert^2
}{2}. 
\end{multline*} 
Using the maximal inequality stated in Lemma \vref{lem: 2.7182}
to bound the variance term, we get
\begin{multline*}
\B{P}_{W_1, \, \dots \, , W_n} \biggl\{ 
\sup_{\theta \in \Theta} \rho_{\theta' \, | \, \theta} \Bigl[  \B{P}_W \Bigl( \wt{f} \bigl( \theta', W \bigr) 
\Bigr)
-\oB{P}_W \Bigl( \lambda^{-1} \psi \bigl[ \lambda \wt{f}(\theta', W) \bigr] 
\Bigr) \Bigr] \biggr\} 
\\ \leq \lambda \Bigl[ (b-a)^2/4 + 2 \log(ek) \lVert W \rVert_{\infty}^2 / 
\beta \Bigr] 
+ \frac{ \beta \lVert \Theta \rVert^2}{2n\lambda}.
\end{multline*}
This provides an upper bound for \eqref{eq:42}. Combining it with the 
upper bound for \eqref{eq:43} gives an upper bound for \eqref{eq:41}
that reads 
\begin{multline*}
\B{P}_{W_1, \, \dots \, , W_n} \biggl\{ 
\sup_{\theta \in \Theta} \bigl( \B{P}_W - \oB{P}_{W} \bigr) 
 \rho_{\theta' \, | \, \theta} \bigl[ f(\theta', W) \bigr] \biggr\}\\  
\leq \frac{(\sqrt{2} + 1) \lambda}{2} 
\Bigl[ (b-a)^2/4 + 2 \log(ek) \lVert W \rVert_{\infty}^2 / 
\beta \Bigr] 
+ \frac{ \beta \lVert \Theta \rVert^2}{2n\lambda}.
\end{multline*}
Choosing 
\[
\lambda = \sqrt{\frac{4 \beta \lVert \Theta \rVert^2 
}{(\sqrt{2} + 1)\bigl[(b-a)^2 + 8 \log(ek) \lVert W \rVert_{\infty}^2 / 
\beta \bigr] n}} 
\] 
gives
\begin{multline*}
\B{P}_{W_1, \, \dots \, , W_n} \biggl\{ 
\sup_{\theta \in \Theta}
\bigl( \B{P}_W - \oB{P}_{W} \bigr) 
 \rho_{\theta' \, | \, \theta} \bigl[ f(\theta', W) \bigr] \biggr\}
\\ \leq
\wt{F}(\beta) \overset{\text{def}}{=} \sqrt{\frac{(\sqrt{2} + 1 )\Bigl( \beta
(b-a)^2 + 8 \log(ek) \lVert W \rVert_{\infty}^2 
\Bigr) \lVert \Theta \rVert^2 }{4n}}.
\end{multline*}
Putting everything together, we obtain
\[
\B{P}_{W_1, \, \dots \, , W_n} \Bigl\{ 
\sup_{\theta \in \Theta}
\bigl( \B{P}_W - \oB{P}_W \bigr) f(\theta, W) \Bigr\} \leq
2 \sqrt{2 \log(k)/\beta} \lVert W \rVert_{\infty} + 
\wt{F}(2^{-p} \beta) + p F,
\]
where $F$ is defined by equation \myeq{eq:F}. 

Let us choose $\beta = 2 n \lVert \Theta \rVert^{-2}$ and 
$p = \bigl\lfloor \log(n/k)/\log(2) \bigr\rfloor$, so that
\[
2^{-p} \beta \leq 4 k \lVert \Theta \rVert^{-2}. 
\]
We get
\begin{multline*}
\B{P}_{W_1, \, \dots \, , W_n} \Bigl\{ 
\sup_{\theta \in \Theta}
\bigl( \B{P}_W - \oB{P}_W \bigr) f(\theta, W) \Bigr\} \\ \leq
\Biggl( \frac{\log(n/k)}{\log(2)} \sqrt{\frac{8 \log(k)}{n}} 
+ 
2 \sqrt{\frac{\log(k)}{n}} \; \Biggr)
\lVert \Theta \rVert \lVert W \rVert_{\infty}  
\\ + 
\sqrt{\frac{(\sqrt{2} + 1 )\Bigl( k (b-a)^2 + 2 \log(ek) 
\lVert W \rVert_{\infty}^2 \lVert \Theta \rVert^2  
\Bigr)}{n}}.
\end{multline*}
The upper deviations from this mean are controled by 
the extension of Hoeffding's bound called the
bounded difference inequality (see section 6.1 and 
theorem 6.2 in \cite{boucheron}). 
It gives with probability at least $1-\delta$
\[
 \sup_{\theta \in \Theta}
\bigl( \B{P}_W - \oB{P}_W \bigr) f(\theta, W) \leq 
\B{P}_{W_1, \, \dots \, , W_n} \Bigl\{ 
\sup_{\theta \in \Theta}
\bigl( \B{P}_W - \oB{P}_W \bigr) f(\theta, W) \Bigr\} +
\sqrt{\frac{2 \log(\delta^{-1})}{n}} (b-a).
\]
This proves the first statement of the lemma. 
To get 
the second one, 
add to the previous inequality 
\[
\B{P}_{W_1, \, \dots \, , W_n} \Bigl\{ 
\Bigl( \oB{P}_W - \B{P}_W \Bigr) f( \theta^*, W) \Bigr\} = 0
\]
to get
\begin{multline*}
\B{P}_{W_1, \, \dots \, , W_n} \Bigl\{ 
\sup_{\theta \in \Theta}
\bigl( \B{P}_W - \oB{P}_W \bigr) \Bigl( f(\theta, W) - f(\theta^*, W) 
\Bigr) \Bigr\} \\ \leq
\Biggl( \frac{\log(n/k)}{\log(2)} \sqrt{\frac{8 \log(k)}{n}} 
+ 
2 \sqrt{\frac{\log(k)}{n}} \; \Biggr)
\lVert \Theta \rVert \lVert W \rVert_{\infty}  
\\ + 
\sqrt{\frac{(\sqrt{2} + 1 )\Bigl( k (b-a)^2 + 2 \log(ek) 
\lVert W \rVert_{\infty}^2 \lVert \Theta \rVert^2  
\Bigr)}{n}},
\end{multline*}
and apply the bounded difference inequality 
to get the deviations.
To prove the end of the proposition concerning an 
estimator $\wh{\theta}$, apply what is already proved 
to the weak closure $\ov{\Theta}$ of $\Theta$ and to 
\[
\theta^* \in \arg \min_{\theta \in \ov{\Theta}} 
\B{P}_W \Bigl( \min_{j \in \dlb 1, \, k \drb} \langle 
\theta_j, W \rangle \Bigr)
\] 
that exists due to Proposition \vref{prop:exists}.
\end{proof}

\section{Generalization bounds for the quadratic $k$-means criterion}

The most obvious application of the previous lemma is to get 
a dimension free bound for the usual quadratic $k$-means criterion.

\begin{prop}
Consider a random vector $X$ in a separable Hilbert space $H$. 
Let \linebreak $(X_1, \dots, X_n)$ be a sample made of $n$ independent copies 
of $X$. Consider the ball of radius $B$
\[
\C{B} = \bigl\{ x \in H \, : \, \lVert x \rVert \leq B \bigr\}
\]
and assume that $\B{P} \bigl( X \in \C{B} \bigr) = 1$
and that $n \geq 2 k$ and $k \geq 2$. 
For any $\delta \in ]0,1[$, with probability at least 
$1 - \delta$, 
\begin{multline*}
\sup_{c \in \C{B}^k} \bigl( \B{P}_X - \oB{P}_X \bigr) \Bigl( \min_{j \in \dlb 1, k \drb} 
\bigl( \lVert c_j \rVert^2 - 2 \langle c_j, X \rangle \bigr) \Bigr) \\ \shoveleft{ 
\qquad 
\leq 
 B^2 \log \Bigl(\frac{n}{k} \Bigr) 
\sqrt{\frac{k \log(k) }{n}}
\Biggl( \underbrace{\rule[-4ex]{0ex}{2ex}\frac{6 \sqrt{2}}{\log(2)}}_{\leq 12.3} + \frac{6}{ \log(n/k)}} \\ \shoveright{
+ \frac{1}{\log(n/k)} \sqrt{\frac{2 \bigl( \sqrt{2} + 1 \bigr) \bigl( 17 + 9 \log(k)\bigr)
}{
\log(k)}}
\; \Biggr)
 + 2 B^2 \sqrt{\frac{2 \log(\delta^{-1})}{n}} }
\\ \leq 16 B^2 \log \Bigl( \frac{n}{k} \Bigr) \sqrt{\frac{
k \log(k)}{n}} 
+ 2 B^2 \sqrt{\frac{2 \log(\delta^{-1})}{n}}.
\end{multline*}
Concerning the excess risk, for any $c^* \in \C{B}^k$,
with probability at least $1 - \delta$, 
\begin{multline*}
\sup_{c \in \C{B}^k} \bigl(\B{P}_X - \oB{P}_X \bigr) \Bigl[ 
\Bigl( \min_{j \in \dlb 1, k \drb} 
\lVert X - c_j \rVert^2 \Bigr) 
- \Bigl( \min_{j \in \dlb 1, k \drb} 
\lVert X - c^*_j \rVert^2 \Bigr) \Bigr]
\\ \leq 
16 \, B^2 \log \Bigl(\frac{n}{k} \Bigr) 
\sqrt{\frac{k \log(k) }{n}}
+ 4 \, B^2 \sqrt{\frac{2 \log(\delta^{-1})}{n}} .
\end{multline*}
Consequently, for any $\epsilon \geq 0$, 
for any $\epsilon$-minimizer
$\wh{c}$, that is for any $\wh{c} \in \C{B}^k$ 
depending on the observed sample and satisfying
\[
\oB{P}_X \Bigl( \min_{j \in \dlb 1, k \drb} 
\lVert X - \wh{c}_j \rVert^2 \Bigr) \leq
\inf_{c \in H^k}
\oB{P}_X \Bigl( \min_{j \in \dlb 1, k \drb} 
\lVert X - c_j \rVert^2 \Bigr) + \epsilon,
\]
for any $\delta \in ]0,1[$,
with probability at least $\ds 1 - \delta $,
\begin{multline*}
\B{P}_X \Bigl( \min_{j \in \dlb 1, k \drb} 
\lVert X - \wh{c}_j \rVert^2 \Bigr) \leq \inf_{c \in H^k}
\B{P}_X \Bigl( \lVert X - c_j \rVert^2 \Bigr) \\
+ 16 \, B^2 \log \Bigl(\frac{n}{k} \Bigr) 
\sqrt{\frac{k \log(k) }{n}}
+ 4 \, B^2 \sqrt{\frac{2 \log(\delta^{-1})}{n}} + \epsilon .
\end{multline*}
Moreover, we also have a bound in expectation with respect 
to the statistical sample distribution: 
\begin{multline*}
\B{P}_{X_1, \dots, X_n} \Bigl[ \B{P}_X \Bigl( \min_{j \in \dlb 1, k \drb}
\lVert X - \wh{c}_j \rVert^2 \Bigr) \Bigr] \leq \inf_{c \in H^k} 
\B{P}_X \Bigl( \lVert X - c_j \rVert^2 \Bigr) 
\\ + 16 \, B^2 \log \Bigl( \frac{n}{k} \Bigr) \sqrt{\frac{k \log(k)}{n}}
+ \epsilon. 
\end{multline*}
\end{prop}
The general meaning of this proposition is that
a chaining argument yields a dimension free non asymptotic
 generalization bound that 
decreases as $\sqrt{k/n}$ 
up to logarithmic factors.

\begin{proof}
We choose to work with the risk function 
\[
\min_{j \in \dlb 1, k \drb} \bigl( \lVert c_j \rVert^2 - 2 \langle
c_j, X \rangle \bigr) = 
\min_{j \in \dlb 1, k \drb} \bigl( \lVert X - c_j \rVert^2 \bigr) - 
\lVert X \rVert^2  
\]
because this provides slightly better constants.
Introduce $W = ( - 2 X, \gamma B) \in H \times \B{R}$
and $\theta_j = ( c_j, \gamma^{-1} \lVert c_j \rVert^2 B^{-1})$,
where the parameter $\gamma > 0$ will be optimized later on.
Remark that
\[
\lVert c_j \rVert^2 - 2 \langle c_j, X \rangle = \langle \theta_j, W \rangle\in [-B^2, 3 B^2 ].
\] 
Note also that 
\[
\lVert W \rVert^2 \lVert \theta_j \rVert^2 \leq B^4 \bigl( 4 + \gamma^2 \bigr)
\bigl( 1 + \gamma^{-2} \bigr) = B^4 \bigl( 5 + \gamma^2 + 4 \gamma^{-2} \bigr)
\] 
and optimize the right-hand size, choosing $\gamma = \sqrt{2}$, to get
\[
\lVert W \rVert^2 \lVert \Theta \rVert^2 \leq 9 k B^4,
\]
where 
\[
\Theta = \bigl\{ \bigl( c_j, 2^{-1/2} B^{-1} 
\lVert c_j \rVert^2 \bigr)_{j=1}^k \in (H \times \B{R})^k \,
: \, c \in \C{B}^k \bigr\}.
\] 
The proposition is then a transcription of 
Lemma \vref{lem:14.1} together with the simplification
\begin{multline}
\label{eq:simple}
\min \Biggl\{ 4,  
\log \Bigl(\frac{n}{k} \Bigr) 
\sqrt{\frac{k \log(k) }{n}}
\Biggl( \frac{6 \sqrt{2}}{\log(2)} + \frac{6}{ \log(n/k)} 
\\ + \frac{1}{\log(n/k)} \sqrt{\frac{2 \bigl( \sqrt{2} + 1 \bigr) \bigl( 17 + 9 \log(k)\bigr)
}{
\log(k)}}
\; \Biggr) \Biggr\}
\\ \leq 16
\log \Bigl(\frac{n}{k} \Bigr) 
\sqrt{\frac{k \log(k) }{n}}
\end{multline}
that holds for any $k \geq 2$ and any $n \geq 2k$ and can be used
since $4B^2$ is a trivial bound.
Remark that in the three last inequalities of the 
proposition we can take the infimum on $c \in H^k$
instead of $c \in \C{B}^k$, since it is in fact reached on $\C{B}^k$.

Thus, all that remains to prove is \eqref{eq:simple}.

Putting $\ds a = \frac{6 \sqrt{2}}{\log(2)}$, $b = 16$, $\rho = n/k$,
\begin{align*}
\eta & = 6 + \sqrt{\frac{2 \bigl( \sqrt{2} + 1 \bigr) \bigl( 
17 + 9 \log(k) \bigr)}{\log(k)}}, \\ 
f(\rho, k) & = \sqrt{\log(k)/\rho} \Bigl(a \log(\rho) + \eta(k) \Bigr) \\ 
\text{and } g(\rho, k) & = b \sqrt{ \log(k) / \rho} \log(\rho),
\end{align*}
we have to prove that
\[
\min \bigl\{ 4, f(\rho, k) \bigr\} \leq g(\rho, k), \qquad \rho \geq 2, 
k \geq 2.
\] 
In other words, we have to prove that, when $g(\rho, k) < f(\rho, k)$,
then $g(\rho, k) \geq 4$. This can also be written as
\[
g(\rho, k) \geq 4, \qquad \min \{ \rho, k \} \geq 2, \; g(\rho, k) < f(\rho, k).
\] 
According to the definitions, this is also equivalent to 
\[
\log(\rho) - 2 \log \bigl( \log(\rho) \bigr) \leq 2 \log(b/4) + \log
\bigl( \log(k) \bigr), \quad \min \{ \rho, k \} \geq 2, \;
(b - a) \log(\rho) \leq \eta(k).
\] 
Since $\eta$ is decreasing and since $k \mapsto \log \bigl( \log
(k) \bigr)$ is increasing, if the statement is true for $k = 2$, it is 
true for any $k \geq 2$. Thus we have to prove that
\[
\log(\rho) - 2 \log \bigl( \log(\rho) \bigr) \leq 2 \log ( b / 4 ) + 
\log \bigl( \log(2) \bigr), \quad \log(2) \leq \log(\rho) \leq \eta(2) / (b-a). 
\]
Putting $\xi = \log(\rho)$, we have to prove that
\[
\xi - 2 \log(\xi) \leq 2 \log(b/4) + \log \bigl( \log(2) \bigr), \quad
\log(2) \leq \xi \leq \eta(2) / (b-a).
\]
Since $\xi \mapsto \xi - 2 \log(\xi)$ is convex, it is enough to 
check the inequality at the two ends of the interval, that is 
when $\xi \in \{ \log(2), \eta(2)/(b-a) \}$, which can be done 
numerically. More precisely, we have to check that
\begin{multline*}
2 \log(b/4) + \log \bigl( \log(2) \bigr) \\ - \max \Bigl\{ 
\log(2) - 2 \log \bigl( \log(2) \bigr), \eta(2)/(b-a) - 2 
\log \bigl[ \eta(2)/(b-a) \bigr] \Bigr\} \geq 0,
\end{multline*}
and we get numerically that the left-hand side is larger than the 
minimum of $0.9$ and $0.6$.
\end{proof}

\section{Generalization bounds for the robust $k$-means
criterion}

\begin{prop}
% \label{prop:12.1}
Let $X$ be a random vector in a separable Hilbert space $H$
and let $(X_1, \dots, X_n)$ be a statistical sample made of 
$n$ independent copies of $X$.
Consider for some scale parameter $\sigma > 0$
the criterion $\C{R}_2$ of equation \myeq{eq:robustQuadratic}
and its empirical counterpart
\[
\oC{R}_2(c) = 2 \sigma^2 \, \oB{P}_X \Bigl[ 1 - \exp \Bigl( 
- \frac{1}{2 \sigma^2} \min_{j \in \dlb 1, \, k \drb} 
\lVert X - c_j \rVert^2 \Bigr) \Bigr], \quad c \in H^k.
\] 
Consider any $k \geq 2$ and any $n \geq 2 k$.
For any $\delta \in ]0,1[$, 
with probability at least $1 - \delta$, for any $c \in 
H^k$,
\begin{multline*}
\C{R}_2(c) \leq 
\oC{R}_2(c) + 2 \sigma^2 \Biggl(  
\frac{\log(n/k)}{\log(2)} \sqrt{\frac{8 \, k \log(k)}{n}} + 
2 \sqrt{\frac{k \log(k)}{n}} 
\\ + 
\sqrt{\frac{(\sqrt{2} + 1 ) \, k \, \bigl( 3 + 2 \log(k) 
\bigr) }{n}} + \sqrt{\frac{\log(\delta^{-1})}{2n}} \; \Biggr).
\end{multline*}
For any non random family of centers $c^* \in H^k$, 
with probability at least $1 - \delta$, for any $c 
\in H^k$,
\begin{multline*} 
\C{R}_2(c) - \C{R}_2(c^*) \leq
\oC{R}_2(c) - \oC{R}_2(c^*) + 2 \sigma^2 \Biggl(
 \frac{\log(n/k)}{\log(2)} 
\sqrt{\frac{8 \, k  \log(k)}{n}} + 
2 \sqrt{\frac{k \log(k)}{n}} 
\\ + 
\sqrt{\frac{(\sqrt{2} + 1 )\, k \, \bigl( 3 + 2 \log(k) 
\bigr) }{n}} + \sqrt{\frac{2\log(\delta^{-1})}{n}} \; \Biggr). 
\end{multline*} 
Consequently, for any $\epsilon \geq 0$,
if $\wh{c}$ is an $\epsilon$-minimizer satisfying 
\[
\oC{R}_2( \wh{c} \, ) \leq \inf_{c \in \B{R}^{d \times k}} 
\oC{R}_2(c) + \epsilon,
\]
with probability at least $1 - \delta$, 
\begin{multline*} 
\C{R}_2 ( \wh{c} \, ) \leq \inf_{c \in \B{R}^{d \times k}} 
\C{R}_2 ( c ) + 2 \sigma^2 \Biggl(
\frac{\log(n/k)}{\log(2)} \sqrt{\frac{8 \,
k \log(k)}{n}} + 
2 \sqrt{\frac{k \log(k)}{n}} 
\\ + 
\sqrt{\frac{(\sqrt{2} + 1 ) \, k \, \bigl( 3 + 2 \log(k) 
\bigr) }{n}} + \sqrt{\frac{2\log(\delta^{-1})}{n}} \; \Biggr)
 + \epsilon.
\end{multline*} 
In the same way, in expectation,
\begin{multline*} 
\B{P}_{X_1, \, \dots \, , X_n} \Bigl( \C{R}_2 ( \wh{c} \, ) \Bigr)
\leq \inf_{c \in \B{R}^{d \times k}} 
\C{R}_2 ( c ) + 2 \sigma^2 \Biggl(
\frac{\log(n/k)}{\log(2)} \sqrt{\frac{8 \,
k \log(k)}{n}} + 
2 \sqrt{\frac{k \log(k)}{n}} 
\\ + 
\sqrt{\frac{(\sqrt{2} + 1 ) \, k \, \bigl( 3 + 2 \log(k) 
\bigr) }{n}} 
\; \Biggr)
 + \epsilon.
\end{multline*} 
\end{prop}
As we can see, the robust criterion has a scale parameter $\sigma$,
that allows to remove all integrability conditions
on the sample distribution or boundedness assumptions on 
the centers.

\begin{proof}
According to the Aronszajn theorem \cite{aronszajn1950theory}, there exists a mapping 
$\Psi : H \rightarrow \C{H}$ such that
\[
\exp \Bigl( - \frac{1}{2 \sigma^2} \lVert x - y \rVert^2 \Bigr) = 
\langle \Psi(x), \Psi(y) \rangle_{\C{H}}, \quad x, y \in H.
\] 
Moreover the reproducing kernel Hilbert space $\C{H}$, being 
based on a continuous kernel defined on a separable topological 
space, is separable according to 
\cite[lemma 4.33 page 130]{christmann}.
We can express the risk as 
\begin{align*} 
\C{R}_2( c ) =
2 \sigma^2 \Bigl[ 1 + \B{P}_W \Bigl( 
\min_{j \in \dlb 1, k \drb} \langle 
- \theta_j, W \rangle_{\C{H}} \Bigr) \Bigr],
\end{align*}
where $\theta_j=\Psi(c_j)$ and $W=\Psi(X)$. 
The proof then follows from Lemma \vref{lem:14.1},
taking into account that $\theta_j$ and $W$ belong to the unit ball 
of $\C{H}$, 
so that 
$ \lVert W \rVert_{\infty} = 1$ and $ \lVert 
\Theta \rVert = \sqrt{k}$. 
\end{proof}

\section{Generalization bounds for the information $k$-means criterion}

In order to apply Lemma \vref{lem:14.1} and obtain a generalization 
bound, 
we are going to linearize the information $k$-means algorithm
presented in section \vref{section:7},
using the kernel trick. 

Let us introduce the separable Hilbert space $H = \{ (f, x) \in \B{L}^2(\nu) 
\times \B{R} \}$ equipped with the inner-product
\[
\langle h,h'\rangle=\langle h_1,h'_{1} \rangle_{\B{L}^2(\nu) } 
+\mu \, h_2 \, h'_{2}, \qquad h=(h_1, h_2), \; h'=(h_1', h_2') \in H,
\]
where $\mu > 0$ is a positive real parameter to be chosen 
afterwards.
The associated norm is 
\[
\lVert (h_1, h_2) \rVert =  \sqrt{\langle h_1,h_{1} 
\rangle_{\B{L}^2(\nu) } + \mu \, h^2_{2}}=
\sqrt{\int h_1^{2} \, \ud \nu + \mu \,  h^2_{2}}\,, 
\qquad h=(h_1,h_2)\in H. 
\]
%Let us put $\theta_j = \bigl( q_j, \C{K}(q_j, 1) \bigr)$ and $W = \bigl( 
%- \log (p_X), 1 \bigr)$. 
Define for any constant $B \in \B{R}_+$
\[ 
\Theta_B =  \Bigl\{ \bigl( q, \C{K}(q, 1) \bigr) \, : \, 
q \in \B{L}^1_{+,1}(\nu) \cap \B{L}^2 ( \nu ), \;  
\tint q^2 \, \ud \nu \leq B^2
\Bigr\} \subset H, 
\]
this definition being justified by the fact that 
\begin{equation}
\label{eq:3.3}
\C{K}(q,1) = \int q \log(q) \, \ud \nu \leq \log \Bigl( \int q^2 \, 
\ud \nu \Bigr) < + \infty
\end{equation}
whenever $\int q^2 \, \ud \nu < + \infty$. 
\begin{lemma}
\label{lem:1.7}
Assume that $\ds \ess \sup_{X} \int \log(p_X)^2 \, \ud \nu < \infty$ 
and $\ds \ess \sup_X \int p_X^2 \, \ud \nu < \infty$. 
Remark first that the smallest information ball containing the 
support of $\B{P}_{\ds p_X}$ has an information radius 
\begin{multline*}
\inf_{q \in \B{L}_{+,1}^1(\nu)} \ess \sup_X \C{K}(q, p_X) 
\leq \ess \sup_X \C{K}(1, p_X) \\ 
= \ess \sup_X \int \log \bigl(p_X^{-1} \bigr) \, \ud \nu 
\leq \ess \sup_X \biggl( \int \log(p_X)^2 \, \ud \nu \biggr)^{1/2} < \infty.
\end{multline*}
Define $\ds B  = \ess \sup_X \biggl( \int p_X^2 \, \ud \nu  \biggr)^{1/2}
\exp \Biggl[ 
\inf_{q \in \B{L}_{+, 1}^1(\nu)} \ess \sup_X \C{K} (q, p_X) \Biggr]
< \infty$ 
and consider the random variable  
\[
W = \bigl( - \log(p_X), \mu^{-1} \bigr) \in H.
\]
The following two minimization problems are equivalent 
\[ 
\inf_{q \in \big(\B{L}_{+,1}^1(\nu)\big)^k} \B{P}_X \, \Bigl( \min_{j \in \dlb 1, k 
\drb} \C{K}(q_j, p_X) \Bigr) = \inf_{\theta \in \Theta_B^k} \B{P}_W   
\Bigl( \, \min_{j \in \dlb 1, k \drb} \langle \theta_j, W \rangle_H \Bigr).
\] 
\end{lemma} 
\begin{proof}
Let $\ds B' = \ess \sup_{X} \biggl( \int p_X^2 \, \ud \nu \biggr)^{1/2}$ 
and $\ds C = \ess \sup_{X} \biggl( \int \log(p_X)^2 \, \ud \nu \biggr)^{1/2}$.
First let us remark that under the hypothesis of the lemma, the information $k$-means criterion is finite.
Indeed,
\begin{multline*} 
\inf_{q \in \bigl( \B{L}_{+,1}^1(\nu) \bigr)^k} 
\B{P}_X \Bigl( \, \min_{j \in \dlb 1, k \drb} 
\C{K} \bigl( q_j, p_X \bigr) \Bigr) \leq \B{P}_X \Bigl( \C{K}(1, p_X) \Bigr) 
\\ = \B{P}_X \biggl( \int \log \bigl( p_X^{-1} \bigr) \, \ud \nu \biggr)  
\leq C < + \infty. 
\end{multline*}
Now, for any measurable classification function 
$\ell:\C{X}\mapsto \dlb 1, k \drb$ for which the criterion is finite,
we know from Lemma \vref{lem:1.6} that $q^{\star, \ell}_j \in 
\B{L}^2(\nu)$ and we can remark that
\[
\C{K}\bigl(q^{\star,\ell}_j, p_X \bigr)=\langle \theta^{\star,\ell}_j, 
W \rangle_H
\]
where $\theta^{\star, \ell}_j =\bigl(q^{\star,\ell}_j,
\C{K}(q^{\star,\ell}_j,1)\bigr)$. Remark that this definition 
is justified by the fact that
\[
\C{K}(q, 1) = \int q \log(q) \, \ud \nu \leq 
\log \biggl( \int q^2 \, \ud \nu \biggr),
\] 
so that $q^{\star, \ell}_j \in \B{L}^2(\nu)$ implies 
that $\C{K} \bigl( q^{\star, \ell}_j, 1 \bigr) < + \infty$.
So, it is sufficient to conclude the proof to show that $\theta^{\star,\ell}_j
\in \Theta_B$. 
As in the proof of Lemma
\vref{lem:1.6},
\[
\int (q^{\star, \ell}_{j})^2 \, \ud \nu \leq Z_{j}^{-2} \; \B{P}_{X \,
| \, \ell(X) = j}  \biggl( \underbrace{
\int \rule[-5mm]{0mm}{3mm} p_X^2 \, \ud \nu}_{\leq {B'}^2} \biggr)
\leq Z_{j}^{-2} \, {B'}^2, \qquad j \in \dlb 1, k \drb.
\]
By Jensen's inequality, for any $j \in \dlb 1, k \drb$,
\begin{multline*}
Z_{j} = 
\sup_{q \in \B{L}_{+,1}^1(\nu)} 
\int q \exp \biggl\{ \B{P}_{X \, | \, \ell(X) = j} 
 \Bigl[ \log (p_X / q) \Bigr] \biggr\} \, \ud \nu \\
\geq 
\sup_{q \in \B{L}_{+,1}^1(\nu)} 
% mianmian
\exp \biggl\{ \B{P}_{X \, | \, \ell(X) = j}  \biggl[ \int q \log(p_X / q) 
\, \ud \nu \biggr]  \biggr\} 
\\ = \exp \Bigl\{ - \inf_{q \in \B{L}_{+,1}^1(\nu)} \B{P}_{X \, | \,
\ell(X) = j}  \Bigl[ 
\C{K} \bigl( q, p_X \bigr) \Bigr] 
\Bigr\} .
\end{multline*}
Hence
\[ 
Z_{j}^{-1}  \leq 
\exp \Bigl\{ \inf_{q \in \B{L}_{+,1}^1(\nu)} \ess \sup_X 
\C{K}(q,p_X) \Bigr\} \leq \exp ( C ).  
\] 
Therefore
\[ 	
\biggl( \int (q^{\star,\ell}_{j})^2 \, \ud \nu \biggr)^{1/2} \leq B' 
\exp \Bigl[ \inf_{q \in \B{L}_{+,1}^1(\nu)} \ess \sup_{X} \C{K}(q, p_X) 
\Bigr] = B \leq B' \exp(C) < \infty,
\] 
proving that $B < \infty$ and that 
$\theta^{\star,\ell}_j = 
\bigl( q^{\star,\ell}_{j}, \C{K}(q^{\star,\ell}_{j}, 1) \bigr) \in \Theta_B$, which concludes the proof. 
\end{proof}

\begin{prop}
\label{prop:7.10} %Olivier
Under the hypotheses of the previous lemma
there exists an optimal quantizer $\theta^\star \in \Theta_B^k$ minimizing the $k$-means risk,
that is such that
\[
\B{E}\Bigl( \, \min_{j \in \dlb 1, k \drb} \langle \theta^{\star}_j, W \rangle \Bigr)=
\inf_{\theta \in \Theta_B^k} \B{E}\Bigl( \, \min_{j \in \dlb 1, k \drb} \langle \theta_j, W \rangle \Bigr).
\]
\end{prop}
\begin{proof}
Note that 
\[
\lVert \Theta_B \rVert = \sup_{\theta \in \Theta_B} \lVert \theta \rVert
\leq \sqrt{ B^2 + \mu \log(B^2)^2} < + \infty,
\] 
according to equation \myeq{eq:3.3}. Therefore $\Theta_B$ is 
bounded. Applying Proposition \vref{prop:exists} to $\Theta_B^k$,
we find $\wt{\theta} \in \ov{\Theta_B^k}$, the weak closure
of $\Theta_B^k$, such that
\[
\C{R}(\wt{\theta}) \overset{\text{\rm def}}{=} \B{P}_W \Bigl( 
\min_{j \in \dlb 1, \, k \drb} \langle \wt{\theta}_j, W \rangle \Bigr)
= \inf_{\theta \in \Theta_B^k} \C{R}(\theta).
\] 
Remark now that, since, 
according to the Donsker Varadhan representation,
\[
\C{K}(q,1) = \sup_{h \in \B{L}^2(\nu)} \int h q \, \ud \nu - 
\log \biggl( \int \exp(h) \, \ud \nu \biggr),
\]
% mianmian
the function $q \mapsto \C{K}(q,1)$ defined on $\B{L}^2(\nu) \cap 
\B{L}_{+,1}^1(\nu)$
is 
weakly lower semicontinuous. Indeed, it is a supremum of 
weakly continuous function. Accordingly, its epigraph is weakly 
closed. As $\Theta_B$ belongs to this epigraph, its weak closer 
also belongs to it. This implies that for each $j \in \dlb 1, k \drb$,
$\wt{\theta}_j$ belongs to it, 
so that $\wt{\theta} = \bigl( (q_j , y_j), 
j \in \dlb 1, k \drb \bigr)$, where $y_j \geq \C{K}(q_j, 1)$.
Indeed the weak closure of $\Theta_B^k$ is the product 
$\ov{\Theta}_B^k$ of $k$ times the weak closure of $\Theta_B$.
Let us put $\theta^* = \Bigl( \bigl(q_j, \C{K}(q_j,1) \bigr), 
j \in \dlb 1, k \drb \Bigr)$.
By monotonicity of $\C{R}$ with respect to $y_j$, the corresponding coefficient 
of $W$ being positive,
\[
\inf_{\theta \in \Theta_B^k} \C{R}(\theta) =  
\C{R}(\wt{\theta}) \geq \C{R}(\theta^*). 
\]
Since $\theta^* \in \Theta_B^k$, the reverse inequality also 
holds and $\ds \C{R}(\theta^*) = \inf_{\theta \in \Theta_B^k}
\C{R}(\theta)$. 
\end{proof}

The link we just made between the information $k$-means criterion and 
the linear $k$-means criterion allows us to apply Lemma \vref{lem:14.1},
proving the next proposition.

\begin{prop}
Assume that 
\[ 
\ess \sup_X \biggl( \int p_X^2 \, \ud \nu \biggr) < + \infty 
\quad \text{ and } \quad \ess \sup_X \biggl( \int \log ( p_X )^2 \, \ud \nu 
\biggr) < + \infty. 
\] 
Consider the information radius 
\[
R = \inf_{q \in \B{L}_{+,1}^1( \nu )} \ess \sup_{X} \C{K} \bigl(
q, p_X \bigr) 
\]
and the bounds
\begin{align*}
B & = \ess \sup_X \biggl( \int p_X^2 \, \ud \nu \biggr)^{1/2} \exp ( R )\\
\text{and } C & = \ess \sup_X \biggl( \int \log(p_X)^2 \, 
\ud \nu \biggr)^{1/2}. 
\end{align*}
Introduce the parameter space
\[
\C{Q}_B=\Bigl\{ q\in \B{L}^1_{+,1}(\nu) \cap \B{L}^2 ( \nu ) \, : \,
\tint q^2 \, \ud \nu \leq B^2 \Bigr\}.
\]
Given $(X_1, \dots, X_n)$, a sample made of $n$ independent copies of $X$,
with probability at least $1 - \delta$, for any $q \in \C{Q}_B^k$,
\begin{multline*}
\bigl( \B{P}_X - \oB{P}_X \bigr) \Bigl( \min_{j \in \dlb 1, k \drb} 
\C{K}(q_j, p_X) 
\Bigr) 
\\ \shoveleft{ \qquad \leq 
\Biggl( \frac{\log(n/k)}{\log(2)} \sqrt{\frac{8 \, k \log(k)}{n}} + 
2 \sqrt{\frac{k \log(k)}{n}}} 
\\ + 
\sqrt{\frac{(\sqrt{2} + 1 ) k \bigl( 3 + 2 \log(k) 
\bigr) }{n}} + \sqrt{\frac{\log(\delta^{-1})}{2n}}
\; \Biggr) \bigl( BC + 2 \log(B) \bigr).
\end{multline*}
For some $\epsilon \geq 0$, consider an empirical $\epsilon$-minimizer $\wh{q}\,(X_1, \dots, X_n) \in \C{Q}_B^k$
satisfying
\[
\oB{P}_X \Bigl( \min_{j \in \dlb 1, k \drb} \C{K} \bigl( \wh{q}_j, p_X \bigr) \Bigr) \leq
\inf_{q \in \C{Q}_B^k} 
\oB{P}_X \Bigl( \min_{j \in \dlb 1, k \drb} \C{K} \bigl( q_j, p_X \bigr) \Bigr)
+ \epsilon.
\]
For any $\delta \in ]0,1[$, with probability at least $1 - \delta$,
\begin{multline*}
\B{P}_X \Bigl( \min_{j \in \dlb 1, k \drb} \C{K} \bigl( \wh{q}, p_X \bigr) 
\Bigr) \leq \inf_{q \in \bigl( \B{L}_{+,1}^1(\nu)
\bigr)^k} \B{P}_X \Bigl( \min_{j \in \dlb 1, 
k \drb} \C{K}\bigl(q_j, p_X \bigr) \Bigr) \\ 
\shoveleft{ \qquad +
\Biggl( \frac{\log(n/k)}{\log(2)} \sqrt{\frac{8 \, k \log(k)}{n}} + 
2 \sqrt{\frac{k \log(k)}{n}}} 
\\ + \sqrt{\frac{(\sqrt{2} + 1 ) \, k \, \bigl( 3 + 2 \log(k) 
\bigr) }{n}} + 
\sqrt{\frac{2\log(\delta^{-1})}{n}} \; \Biggr) \bigl(
BC + 2 \log(B) \bigr) + \epsilon.
\end{multline*}
Moreover, in expectation, 
\begin{multline*}
\B{P}_{X_1, \dots, X_n} \Bigl[ \B{P}_X \Bigl( \min_{j \in \dlb 1, k \drb} \C{K} \bigl( \wh{q}, p_X \bigr) 
\Bigr) \Bigr] \leq \inf_{q \in \bigl( \B{L}_{+,1}^1(\nu)
\bigr)^k} \B{P}_X \Bigl( \min_{j \in \dlb 1, 
k \drb} \C{K}\bigl(q_j, p_X \bigr) \Bigr) \\ 
\shoveleft{ \qquad +
\Biggl( \frac{\log(n/k)}{\log(2)} \sqrt{\frac{8 \, k \log(k)}{n}} + 
2 \sqrt{\frac{k \log(k)}{n}}} 
\\ + \sqrt{\frac{(\sqrt{2} + 1 ) \, k \, \bigl( 3 + 2 \log(k) 
\bigr) }{n}} 
\; \Biggr) \bigl(
BC + 2 \log(B) \bigr) + \epsilon.
\end{multline*}
\end{prop}
\begin{proof}
Apply Lemma \vref{lem:1.7}.
Note that, choosing $\mu = B C^{-1} \log(B^2)^{-1}$,
we get 
\[
\lVert \Theta_B^k \rVert^2 \lVert W \rVert_{\infty}^2 
\leq k \bigl( B^2 + \mu \log(B^2)^2 \bigr) \bigl( C^2 + \mu^{-1} \bigr)
= k \bigl( BC + 2 \log(B) \bigr)^2.
\] 
Remark also that
for any $\theta \in \Theta_B^k$, with probability one,
\[
\min_{j \in \dlb 1, k \drb} \langle \theta_j, W \rangle 
\in \bigl[ 0, \lVert \Theta_B \rVert \lVert W \rVert_{\infty} \bigr] 
\subset \bigl[ 0, BC + 2 \log(B) \bigr]. 
\]
Use these bounds in Lemma \vref{lem:14.1} to conclude the proof.
\end{proof}

\section*{Acknowledgements}

We are grateful to Nikita Zhivotovskiy for useful
comments and references.

\bibliographystyle{imsart-number} % Style BST file (imsart-number.bst or imsart-nameyear.bst)
\bibliography{ref11.bib}       % Bibliography file (usually '*.bib')

%% or include bibliography directly:
% \begin{thebibliography}{}
% \bibitem{b1}
% \end{thebibliography}

\end{document}